\documentclass{amsart}

\usepackage{amssymb}
\usepackage{euscript}
\usepackage{mathrsfs}

\usepackage{amscd}
\usepackage[cmtip,all]{xy}
\usepackage{graphicx}

\newtheorem{theorem}{Theorem}[section]
\newtheorem{proposition}[theorem]{Proposition}
\newtheorem{lemma}[theorem]{Lemma}
\newtheorem{corollary}[theorem]{Corollary}
\newtheorem{conjecture}[theorem]{Conjecture}

\theoremstyle{remark}
\newtheorem{remark}[theorem]{Remark}

\newtheorem{example}[theorem]{Example}
\newtheorem{examples}[theorem]{Examples}

\numberwithin{equation}{section}

\DeclareMathOperator{\Cl}{Cl}
\DeclareMathOperator{\disc}{disc}
\DeclareMathOperator{\diag}{diag}
\DeclareMathOperator{\End}{End}

\DeclareMathOperator{\I}{\mathbf{I}}
\DeclareMathOperator{\Jac}{Jac}
\DeclareMathOperator{\pc}{cp}
\DeclareMathOperator{\Real}{Re}
\DeclareMathOperator{\sgn}{sgn}
\DeclareMathOperator{\sym}{sym}
\DeclareMathOperator{\Trace}{Trace}
\DeclareMathOperator{\vol}{volume}

\DeclareMathOperator{\Mat}{\mathbf{M}}
\DeclareMathOperator{\GL}{\mathbf{GL}}
\DeclareMathOperator{\SO}{\mathbf{SO}}
\DeclareMathOperator{\SU}{\mathbf{SU}}
\DeclareMathOperator{\USp}{\mathbf{USp}}

\newcommand{\CC}{\mathbb{C}}
\newcommand{\FF}{\mathbb{F}}

\newcommand{\NN}{\mathbb{N}}
\newcommand{\PP}{\mathbb{P}}
\newcommand{\QQ}{\mathbb{Q}}
\newcommand{\RR}{\mathbb{R}}
\newcommand{\TT}{\mathbb{T}}

\newcommand{\ZZ}{\mathbb{Z}}

\newcommand{\card}[1]{\lvert#1\rvert}
\newcommand{\h}{h}
\newcommand{\Legendre}[2]{\EuScript{P}_{#1}^{#2}}
\newcommand{\Meijer}[6]{
\renewcommand{\arraystretch}{1.2}
G^{#1,#2}_{#3,#4}\left(#5\left|#6\right.\right)}
\newcommand{\set}[2]{\left\{#1\,\mid\,#2\right\}}
\newcommand{\vaf}{\mathsf{F}}
\newcommand{\vat}{\boldsymbol{\tau}}

\newcommand{\VAA}{\mathscr{C}(\bar{A}_{g})}
\newcommand{\VAC}{\mathscr{C}(X_{g})^{\sym}}
\newcommand{\VAG}{\mathscr{C}(G)^{\circ}}
\newcommand{\VAI}{\mathscr{C}(I_{g})^{\sym}}
\newcommand{\VAS}{\mathscr{C}(\Sigma_{g})}

\begin{document}

\title[trace in the unitary symplectic group]{On the distribution of the trace \\
in the unitary symplectic group \\
and the distribution of Frobenius }

\author{Gilles Lachaud}

\address{Aix Marseille Universit\'e, CNRS, Centrale Marseille, I2M, UMR 7373, 13453 Marseille, France}

\email{gilles.lachaud@univ-amu.fr}



\begin{abstract}
The purpose of this article is to study the distribution of the trace on the unitary symplectic group. We recall its relevance to equidistribution results for the eigenvalues of the Frobenius in families of abelian varieties over finite fields, and to the limiting distribution of the number of points of curves. We give four expressions of the trace distribution if $g = 2$, in terms of special functions, and also an expression of the distribution of the trace in terms of elementary symmetric functions. In an appendix, we prove a formula for the trace of the exterior power of the identity representation.
\end{abstract}

\maketitle

\setcounter{tocdepth}{1}
\tableofcontents


\section{Introduction}

Let $G$ be a connected compact Lie group, and $\pi: G \longrightarrow \GL(V)$ a continuous representation of $G$ on a finite dimensional complex vector space $V$. The map
$$m \mapsto \vat(m) = \Trace \pi(m)$$
is a continuous central function on $G$,  whose values lie in a compact interval $I \subset \RR$. The \emph{distribution} or \emph{law} of $\vat$ is the measure  $\mu_{\vat} = \vat_{*}(dm)$ on $I$ which is the image by $\tau$ of the mass one Haar measure $dm$ on $G$. That is, for any continuous real function $\varphi \in \mathscr{C}(I)$, we impose the integration formula
$$
\int_{I} \varphi(x) \,\mu_{\vat}(x) =
\int_{G} \varphi(\Trace \pi(m)) \, dm.
$$
Alternately, if $x \in \RR$, then
$$
\mathrm{volume}\set{m \in G}{\Trace \pi(m) \leq x} = \int_{- \infty}^{x} \mu_{\vat}.
$$

We are especially interested here with the group $G = \USp_{2g}$ of symplectic unitary matrices of order $2g$, with $\pi$ equal to the identity representation in $V = \CC^{4}$. With the help of Weyl's integration formula, one establishes that the distribution $\mu_{\vat}$ has a \emph{density} $f_{\vat}$, that is, a positive continuous function such that $d\mu_{\vat}(x) = f_{\vat}(x) dx$. Our main purpose is the study of $f_{\vat}$, especially in the case $g = 2$  and $g = 3$. For instance, for $g = 2$, we have
$$
f_{\vat}(x) = \frac{1}{4 \pi} \left(1 - \frac{x^{2}}{16}\right)^{4} \ {_{2}F_{1}}\left(\frac{3}{2}, \frac{5}{2} ; 5 ; 1 - \frac{x^{2}}{16} \right) 	
$$
if $\card{x} \leq 4$, with \emph{Gauss' hypergeometric function} ${_{2}F_{1}}$ (see Theorem \ref{hypergeometric}).

Another representation of the distribution of the trace, following a program of Kohel, is realized by the \emph{Vi\`ete map}, which is the polynomial mapping
$$
\mathsf{s}(t) = (s_{1}(t), \dots, s_{g}(t)), \quad t = (t_{1}, \dots, t_{g}),
$$
where $s_{n}(t)$ is the elementary symmetric polynomial of degree $n$. Let $I_{g} = [-2,2]^{g}$. The \emph{symmetric alcove} is the set
$$
\Sigma_{g} = \mathsf{s}(I_{g}) \subset \RR^{g},
$$
which is homeomorphic to the $g$-dimensional simplex. By a change of variables in Weyl's integration formula, one obtains a measure $\alpha_{x}$ on the hyperplane section 
$$V_{x} = \set{s \in \Sigma_{g}}{s_{1} = x}$$
such that, if $\card{x} < 2 g$,
$$
f_{\vat}(x) = \int_{V_{x}} \alpha_{x}(s).
$$
As a motivation for the study of these distributions, it is worthwhile to recall that they provide an answer to the following question:
\begin{center}
\emph{Can one predict the number of points of a curve of given genus over a finite field?}
\end{center}
When a curve $C$ runs over the set $\mathsf{M}_{g}(\FF_{q})$ of $\FF_{q}$-isomorphism classes of (nonsingular, absolutely irreducible) curves of genus $g$ over $\FF_{q}$, the number $\card{C(\FF_{q})}$ seems to vary at random. According to Weil's inequality, an accurate approximation to this number is close to $q + 1$, with a normalized ``error'' term $\vat(C)$ such that
$$
\card{C(\FF_{q})} = q + 1 - q^{1/2} \, \vat(C), \quad \card{\vat(C)} \leq 2g.
$$
The random matrix model developed by Katz and Sarnak gives many informations on the behaviour of the distribution of $\vat(C)$ on the set $\mathsf{M}_{g}(\FF_{q})$. For instance, according to their theory, and letting $g$ be fixed, for every $x \in \RR$, we have, as $q \rightarrow \infty$ (cf. Corollary \ref{CorEquiNbPts}):
$$
\frac{\card{\set{C \in \mathsf{M}_{g}(\FF_{q})}{\vat(C) \leq x}}}{\card{\mathsf{M}_{g}(\FF_{q})}} =
\int_{- \infty}^{x} f_{\vat}(s) ds + O\left(q^{-1/2}\right).
$$
Hence, the knowledge of $f_{\vat}$ provides a precise information on the behaviour of the distribution of the number of points of curves.

The outline of this paper is as follows. After Section \ref{sec_intro}, devoted to notation, we recall in Section \ref{sec_WeilInt} the Weyl's integration formula, expressed firstly in terms of the angles $(\theta_{1}  \dots, \theta_{g})$ defining a conjugacy class, and secondly in terms of the coefficients $t_{j} = 2 \cos \theta_{j}$. We discuss equidistribution results for a family of curves or abelian varieties over a finite field in Section \ref{sec_equi}. In Section \ref{sec_expressions} we obtain four explicit formulas for the trace distribution if $g = 2$, respectively in terms of hypergeometric series, of Legendre functions, of elliptic integrals, and of Meijer $G$-functions. We also give the distribution of the trace for the representation of the group $\SU_{2} \times \SU_{2}$ in $\USp_{4}$.

In the second part of the paper, we take on a different point of view by using elementary symmetric polynomials, and obtain a new expression of Weyl's integration formula. Section \ref{sec_VieteMap} defines the Vi\`ete map, asssociating to a sequence of coordinates the coefficients of the polynomial admitting as roots the elements of this sequence, and Section \ref{sec_alcove} describes the symmetric alcove, that is, the image of the set of ``normalized real Weil polynomials'' by the Vi\`ete map. By a change of variables using the Vi\`ete map, we obtain in Section \ref{sec_symint} a new integration formula, which leads to another expression for the distribution of the trace on the conjugacy classes, in the cases $g = 2$ and $g = 3$. If $g = 2$, we compute also the trace of $\wedge^{2} \pi$. Finally, we include an appendix on the character ring of $\USp_{2 g}$, including a formula on the exterior powers of the identity representation.

I would like to thank David Kohel for fruitful conversations. Also, I warmly thank the anonymous referee for carefully reading this work and for its suggestions, especially regarding the appendix.

\section{The unitary symplectic group}
\label{sec_intro}

The \emph{unitary symplectic group} $G = \USp_{2 g}$ of order $2 g$ is the real Lie group of complex symplectic matrices
$$
G =
\set{m \in \GL_{2 g}(\CC)}{{^{t}\!m}.J.m = J \ \text{and} \ {^{t}\!m}.\bar{m} = \I_{2 g}},
$$
with
$$J = \begin{pmatrix} 0 & \I_{g} \\ -\I_{g} & 0 \end{pmatrix}.$$
Alternately, the elements of $G$ are the matrices
$$
m = \begin{pmatrix} a & - \bar{b}~ \\ b & \bar{a}~ \end{pmatrix} \in \SU_{2g}, \quad
a,b \in \Mat_{g}(\CC).
$$
The torus $\TT^{g} = (\RR/2 \pi \ZZ)^{g}$ is embedded into $G$ by the homomorphism
\begin{equation}
\label{exptore}
\theta = (\theta_{1}, \dots, \theta_{g}) \mapsto
\h(\theta) =
\begin{pmatrix}
\begin{matrix}
e^{i \theta_{1}} & \dots & 0 \\ 
\dots            & \dots & \dots \\
0                & \dots & e^{i \theta_{g}}
\end{matrix}
& 0 \\ 0 &
\begin{matrix}
e^{- i \theta_{1}} & \dots & 0 \\ 
\dots             & \dots & \dots \\
0                 & \dots & e^{- i  \theta_{g}}
\end{matrix}
\end{pmatrix}
\end{equation}
whose image $T$ is a maximal torus in $G$.  The Weyl group $W$ of $(G,T)$ is the semi-direct product of the symmetric group $\mathfrak{S}_{g}$ in $g$ letters, operating by permutations on the $\theta_{j}$, and of the group $N$ of order $2^{g}$ generated by the involutions $\theta_{j} \mapsto - \theta_{j}$. Since every element of $\USp_{2g}$ has eigenvalues consisting of $g$ pairs of complex conjugate numbers of absolute value one, the quotient $T/W$ can be identified with the set $\Cl G$ of conjugacy classes of $G$, leading to a homeomorphism
$$
\begin{CD}
\TT^{g} / W & @>{\sim}>> & T/W & @>{\sim}>> & \Cl G
\end{CD}
$$

\begin{remark}
\label{Phi2g}
Here is a simple description of the set $\Cl G$. Let $\Phi_{2 g}$ be the subset of monic polynomials $p \in \RR[u]$ of degree $2g$, with $p(0) = 1$, with roots consisting of $g$ pairs of complex conjugate numbers of absolute value one. If $\theta \in \TT^{g}$, let
$$
p_{\theta}(u) = \prod_{j = 1}^{g} (u - e^{i \theta_{j}}) (u - e^{- i \theta_{j}}).
$$
The map $\theta \mapsto p_{\theta}$ is a bijection from $\TT^{g}/W$ to $\Phi_{2 g}$. Renumbering, we may assume that
$$
0 \leq \theta_{g} \leq \theta_{g - 1} \leq \dots < \theta_{1} \leq \pi.
$$
The map
$m \mapsto \pc_{m}(u) = \det(u.\I - m)$ induces a homeomorphism
$$
\begin{CD}
\Cl G & @>{\sim}>> & \Phi_{2 g}
\end{CD}
$$
with $\pc_{m} = p_{\theta}$ if and only if $m$ is conjugate to $\h(\theta)$. The polynomial $p_{\theta}$ is \emph{palindromic}, that is, if
$$
p_{\theta}(u) = \sum_{n = 0}^{2 g} (- 1)^{n} a_{n}(\theta) u^{n},
$$
then  $a_{2 g - n}(\theta) = a_{n}(\theta)$ for $0 \leq n \leq g$.
\end{remark}

\section{Weyl's integration formula}
\label{sec_WeilInt}

The box $X_{g} = [0,\pi]^{g}$ is a fundamental domain for $N$ in $\TT^{g}$ and the map $F \mapsto F \circ 
 \h$ defines an isomorphism
\begin{equation}
\label{BasicIsom}
\begin{CD}
\VAG & @>{\sim}>> & \VAC
\end{CD}
\end{equation}
from the vector space $\VAG = \mathscr{C}(\Cl G)$ of \emph{complex central continuous functions} on $G$ to the space $\VAC$ of complex symmetric continuous functions on $X_{g}$. Notice that the isomorphism \eqref{BasicIsom} has an algebraic analog, namely the isomorphism \eqref{BasicIsomPoly} in the appendix. Let $dm$ be the Haar measure of volume $1$ on $G$. If $\vaf \in \VAG$, then
$$
\int_{\Cl G} \vaf(\Dot{m}) \, d\Dot{m} = \int_{G} \vaf(m) \, dm,
$$
where $d\Dot{m}$ is the image measure on $\Cl G$ of the measure $dm$. The following result is classical \cite[Ch. 9, \S{} 6, Th. 1, p. 337]{BkiLIE79en}, \cite[5.0.4, p. 107]{KS}.

\begin{theorem}[Weyl integration formula, I]
\label{WIF1}
If $\vaf \in \VAG$, then
$$ 
\int_{G} \vaf(m) \, dm = 
\int_{X_{g}} \vaf \circ \h(\theta))\mu_{g}(\theta),
$$
with the \emph{Weyl measure}
$$
\mu_{g}(\theta) = \delta_{g}(\theta) d\theta, \quad d\theta = d\theta_{1} \dots d\theta_{g},
$$
$$
\delta_{g}(\theta) = \frac{1}{g!} \prod_{j = 1}^{g} \left(\frac{2}{\pi} \right) \left(\sin \theta_{j}\right)^{2}
\prod_{j < k} \left(2 \cos \theta_{k} - 2 \cos \theta_{j}\right)^{2}. \rlap \qed
$$
\end{theorem}
We call the open simplex
\begin{equation}
\label{alctheta}
A_{g} = \set{(\theta_{1}, \dots, \theta_{g}) \in X_{g}}
{0 < \theta_{g} < \theta_{g - 1} < \dots < \theta_{1} < \pi}
\end{equation}
the \emph{fundamental alcove} in $X_{g}$. The closure $\bar{A}_{g}$ of $A_{g}$ is a fundamental domain for $\mathfrak{S}_{g}$ in $X_{g}$, and, for every $f \in \VAC$, we have
$$ 
\int_{X_{g}} f(\theta) \, d\theta = g! \int_{A_{g}} f(\theta) \, d\theta.
$$
There is another way to state Weyl's integration formula, which will be used in Section \ref{sec_symint}. Let $I_{g} = [-2,2]^{g}$. The map
$$
(\theta_{1}, \dots, \theta_{g}) \mapsto (2 \cos \theta_{1}, \dots, 2 \cos \theta_{g})
$$
defines an homeomorphism $X_{g} \longrightarrow I_{g}$. Let
$$
k(t_{1}, \dots, t_{g}) = h\left(\arccos\frac{t_{1}}{2}, \dots, \arccos\frac{t_{g}}{2}\right).
$$
where $\h(\theta)$ is given by \eqref{exptore}. Then the map $F \mapsto F \circ k$ defines an isomorphism
$$
\begin{CD}
\VAG & @>{\sim}>> & \VAI
\end{CD}
$$
where $\VAI$ is the space of complex symmetric continuous function on $I_{g}$. For an algebraic analog, see the isomorphism \eqref{SecondIsomPoly} in the appendix. Let
\begin{equation}
\label{D0D1def}
D_{0}(t) = \prod_{j < k}(t_{k} - t_{j})^{2}, \quad
D_{1}(t) = \prod_{j = 1}^{g} (4 - t_{j}^{2}).
\end{equation}

\begin{proposition}[Weyl integration formula, II]
\label{weylalg}
If $\vaf \in \VAG$, then
$$
\int_{G} \vaf(m) \, dm =
\int_{I_{g}} \vaf \circ k(t) \lambda_{g}(t) \,  dt,
$$
where $t = (t_{1}, \dots, t_{g})$ and $dt = dt_{1} \dots dt_{g}$, with the \emph{Weyl measure}
$$
\lambda_{g}(t) dt, \quad dt = dt_{1} \dots dt_{g},
$$
$$
\lambda_{g}(t) = \frac{1}{(2 \pi)^{g} g!} D_{0}(t) \sqrt{D_{1}(t)}.
$$
\end{proposition}

\begin{proof}
If $\varphi \in \mathscr{C}(I_{g})$, we have
$$
\int_{X_{g}}
\varphi(2 \cos \theta_{1}, \dots, 2 \cos \theta_{g}) \delta_{g}(\theta) \,  d\theta
= \int_{I_{g}} \varphi(t) \lambda_{g}(t) dt.
$$
Apply Weyl's integration formula of Theorem \ref{WIF1}.
\end{proof}

As in \eqref{alctheta}, we call the open simplex
\begin{equation}
\label{alcoveIg}
A_{g} = \set{t \in I_{g}}
{- 2 < t_{1} < t_{2} < \dots < t_{g} < 2}
\end{equation}
the \emph{fundamental alcove in $I_{g}$}. Then $\bar{A}_{g}$ is a fundamental domain of $I_{g}$ for $\mathfrak{S}_{g}$, and if $f \in \VAI$, we have
\begin{equation}
\label{weylalgalc}
\int_{I_{g}} f(t) \,  dt
= g! \int_{A_{g}} f(t) \,  dt.
\end{equation}

\begin{examples}
We have
$$
\lambda_{2}(t) =
\frac{1}{4 \pi^{2}} (t_{1} - t_{2})^{2}
\sqrt{(4 - t_{1}^{2}) (4 - t_{2}^{2})}.
$$
The maximum of $\lambda_{2}$ in $A_{2}$ is attained at the point
$$
t_{0} = (- \sqrt{2}, \sqrt{2}) , \quad \text{with} \quad \lambda_{2}(t_{0}) = \frac{2}{\pi^{2}} \, .
$$
We have also
$$
\lambda_{3}(t) =
\frac{1}{48 \pi^{3}} (t_{1} - t_{2})^{2}(t_{1} - t_{3})^{2}(t_{2} - t_{3})^{2}
\sqrt{(4 - t_{1}^{2}) (4 - t_{2}^{2}) (4 - t_{3}^{2})}.
$$
The maximum of $\lambda_{3}$ in $A_{3}$ is attained at the point
$$
t_{0} = (- \sqrt{3}, 0, \sqrt{3}) , \quad \text{with} \quad \lambda_{3}(t_{0}) = \frac{9}{2 \pi^{3}} \, .
$$
\end{examples}

Now, for the convenience of the reader, we recall some notation on the distribution of central functions. Let $G$ be a connected compact Lie group. The Haar measure $dm$ of volume $1$ on $G$ is a probability measure, and $G$ becomes a probability space; \textit{ipso facto}, its elements become random matrices, and the functions in $\VAG$ are complex random variables on $G$. If $\vaf \in \VAG$ is a real random variable, whose values lie in the compact interval $I \subset \RR$, the \emph{distribution} or \emph{law} of $\vaf$ is the image measure $\mu_{\vaf} = F_{*}dm$ on $I$ such that
$$
\int_{I} \varphi(x) \, \mu_{\vaf}(x) =
\int_{G} \varphi(\vaf(m)) \, dm
\quad \text{if} \quad \varphi \in \mathscr{C}(I),
$$
If $B$ is a borelian subset of $I$, then
$$
\mu_{\vaf}(B) = \vol \set{m \in G}{\vaf(m) \in B},
$$
and the \emph{cumulative distribution function} of $\vaf$ is
$$
\Phi_{\vaf}(x) = \PP(\vaf \leq x) = 
\int_{- \infty}^{x} \mu_{\vaf}(t) =
\int_{\vaf(m) \leq x} \, dm.
$$
The \emph{characteristic function} of $\vaf$ is the Fourier transform of $\mu_{\vaf}(x)$:
$$
\widehat{f}_{\vaf}(t) = \int_{- \infty}^{\infty} e^{i t x} \mu_{\vaf}(x) =
\int_{G} e^{i t \vaf(m)} \, dm =
\int_{X_{g}} e^{i t \vaf \circ h(\theta)} \mu_{g} (\theta).
$$
This is an entire analytic function of $t$, of exponential type, bounded on the real line. The distribution $\mu_{\vaf}$ has a \emph{density} if $\mu_{\vaf}(x) = f_{\vaf}(x) dx$ with a positive function $f_{\vaf}$ in $L^{1}(I)$. If $\mu_{\vaf}$ has a density, and if Fourier inversion holds, then
$$
f_{\vaf}(x) = \frac{1}{2 \pi} \int_{- \infty}^{\infty} \widehat{f}_{\vaf}(t) e^{- i t x} \, dt.
$$
Conversely, if $\widehat{f}_{\vaf} \in L^{1}(\RR)$, then $\mu_{\vaf}$ has a density. If $G = \USp_{2 g}$, notice that Weyl's integration formulas supply the \emph{joint probability density function} for the random variables $(\theta_{1}, \dots, \theta_{g})$ and $(t_{1}, \dots, t_{g})$.

The distribution $\mu_{\vaf}$ is characterized by the sequence of its \emph{moments}
$$
M_{n}(\vaf) = \int_{I} x^{n} \, \mu_{\vaf}(x) = \int_{G} \vaf(m)^{n} \, dm, \qquad n \geq 1,
$$
and the characteristic function is a generating function for the moments:
\begin{equation}
\label{CaracMoments}
\widehat{f}_{\vaf}(t) =
\sum_{n = 0}^{\infty} M_{n}(\vaf) \frac{(i t)^{n}}{n!} \, .
\end{equation}

\begin{remark}
\label{standard}
If $\pi$ is an irreducible representation of $G$, with real character $\vat_{\pi}$, then the random variable $\vat_{\pi}$ is \emph{standardized}, \textit{i.e.} the first moment (the mean) is equal to zero and the second moment (the variance) is equal to one.
\end{remark}

\begin{remark}
\label{Leray}
Under suitable conditions, an expression of the density by integration along the fibers can be given. For instance, let $G = \USp_{2 g}$, let $\vaf \in \VAG$ be a $C^{\infty}$ function, and put $J = \vaf \circ h(U)$, where $U$ is the open box $]0, \pi[^{g}$. If $\vaf \circ h$ is a submersion on $U$, and if $x \in J$, then 
$$V_{x} = \set{\theta \in U}{\vaf \circ h(\theta) = x}$$
is a hypersurface. Let $\alpha_{x}$ be the \emph{Gelfand-Leray differential form} on $V_{x}$, defined by the relation
$$d(\vaf \circ h) \wedge \alpha_{x} = \delta_{g}(\theta) d\theta_{1} \wedge \dots \wedge d\theta_{g}.$$
For instance,
$$
\alpha_{x} = (-1)^{j - 1} (\partial(\vaf \circ h)/\partial \theta_{j} )^{-1} \delta_{g}(\theta)
d\theta_{1} \wedge \dots d\theta_{j - 1} \wedge d\theta_{j + 1} \dots \wedge d\theta_{g}
$$
if the involved partial derivative is $\neq 0$. Then the distribution is computed by slicing: since the cumulative distribution function is 
$$
\Phi_{\vaf}(x) =
\int_{\vaf \circ h(\theta) \leq x} \, \mu_{g} (\theta) =
\int_{- \infty}^{x} ds \int_{V_{s}} \alpha_{s} (\theta),
$$
we have
$$
f_{\vaf}(x) = \int_{V_{x}} \alpha_{x} (\theta).
$$
See \cite[Lemma 7.2]{Arnold} and \cite[Lemma 8.5]{Serre2012}.
\end{remark}

\section{Equidistribution}
\label{sec_equi}

Let $A$ be an abelian variety of dimension $g$ over $\FF_{q}$. The \emph{Weil polynomial} of $A$ is the characteristic polynomial $L(A, u) = \det(u.\I - F_{A})$ of the Frobenius endomorphism $F_{A}$ of $A$, and the \emph{unitarized Weil polynomial} of $A$ is
$$
\bar{L}(A,u) = L(A, q^{-1/2}u) =
\prod_{j = 1}^{g} (u - e^{i \theta_{j}}) (u - e^{- i \theta_{j}}).
$$
This polynomial has coefficients in $\ZZ$, belongs to the set $\Phi_{2 g}$ defined in Remark \ref{Phi2g}, and $\theta(A) = (\theta_{1}, \dots \theta_{g})$ is the \emph{sequence of Frobenius angles} of $A$. We write
$$
\bar{L}(A,u) = \sum_{n = 0}^{2 g} (- 1)^{n} a_{n}(A) u^{n},
$$
keeping in mind that $a_{2 g - n}(A) = a_{n}(A)$ for $0 \leq n \leq 2g$, since $\bar{L}(A,u) \in \Phi_{2 g}$. By associating to $A$ the polynomial $\bar{L}(A,u)$, each abelian variety defines, as explained in Section \ref{sec_intro}, a unique class $\dot{m}(A)$ in $\Cl G$, such that
$$
\bar{L}(A,u) = \det(u \I - \dot{m}(A)).
$$

Let $\mathsf{A}_{g}(\FF_{q})$ be the finite set of $k$-isomorphism classes of principally polarized abelian varieties of dimension $g$ over $k$. The following question naturally arises:
\begin{center}
\emph{As $q \rightarrow \infty$, and as $A$ runs over $\mathsf{A}_{g}(\FF_{q})$,
what are the limiting distributions of the random variables $a_{1}, \dots, a_{g}$ ?}
\end{center}
In order to clarify this sentence, we look in particular to the coefficient $a_{1}$, and focus on the Jacobians of curves. Let $C$ be a (nonsingular, absolutely irreducible, projective) curve over $\FF_{q}$. The \emph{Weil polynomial} $L(C,u)$ of $C$ is the Weil polynomial of its Jacobian, and similarly for the \emph{unitarized Weil polynomial} $\bar{L}(C,u)$, the \emph{sequence of Frobenius angles} $\theta(C)$, the coefficients $a_{n}(C)$, and the conjugacy class $\dot{m}(C)$. If $F_{C}$ is the geometric Frobenius of $C$, then
$$
\bar{L}(C,u) = \det(u.\I - q^{-1/2} F_{C}) = \det(u \I - \dot{m}(C)).
$$
Then
\begin{equation}
\label{Weil}
\card{C(\FF_{q})} = q + 1 - q^{1/2} \vat(C),
\end{equation}
where $\vat(C) = a_{1}(C)$, namely
$$
\vat(C) = q^{-1/2} \Trace F_{C} = 2 \sum_{j = 1}^{g} \cos \theta_{j},
$$
with $\theta(C) = (\theta_{1}, \dots, \theta_{g})$.

Then \emph{Katz-Sarnak theory} \cite{KS} models the behavior of the Weil polynomial of a random curve $C$ of genus $g$ over $\FF_{q}$ by postulating that when $q$ is large, the class $\dot{m}(C)$ behaves like a random conjugacy class in $\Cl G$, viewed as a probability space, endowed with the image $d\dot{m}$ of the mass one Haar measure. Here is an illustration of their results. Let $R(G)$ be the character ring of $G$ (cf. the appendix) and
$$
\mathscr{T}(G)^{\circ} = R(G) \otimes \CC \simeq \CC[2 \cos \theta_{1}, \dots 2 \cos \theta_{g}]^{\sym}
$$
the algebra of \emph{continuous representative central functions} on $G$, the isomorphism coming from Proposition \ref{Chevalley}. This algebra is dense in $\VAG$, hence, suitable for testing equidistribution on $\Cl G$. We use the following notation for the average of a complex function $f$ defined over a finite set $\mathsf{Z}$:
$$
\oint_{\ \mathsf{Z}} f(z) dz = 
\frac{1}{\card{\mathsf{Z}}} \sum_{z \in \mathsf{Z}} f(z).
$$
For every finite field $k$, we denote by $\mathsf{M}_{g}(k)$ the finite set of $k$-isomorphism classes of curves of genus $g$ over $k$. The following theorem follows directly, if $g \geq 3$, from \cite[Th. 10.7.15]{KS} (with a proof based on universal families of curves with a $3K$ structure), and from \cite[Th. 10.8.2]{KS} if $g \leq 2$ (with a proof based on universal families of hyperelliptic curves).

\begin{theorem}[Katz-Sarnak]
\label{KS}
Assume $g \geq 1$. If $C$ runs over $\mathsf{M}_{g}(\FF_{q})$, the conjugacy classes $\Dot{m}(C)$ become equidistributed in $\Cl G$ with respect to $d\Dot{m}$ as $q \rightarrow \infty$. More precisely, if $\vaf \in \mathscr{T}(G)^{\circ}$, then
$$
\oint_{\mathsf{M}_{g}(\FF_{q})} \vaf(\dot{m}(C))  dC =
\int_{\Cl G} \vaf(m) \, dm + O\left(q^{-1/2}\right).
\rlap \qed
$$
\end{theorem}

Theorem \ref{KS} means that the counting measures
$$
\mu_{g,q} = \frac{1}{\card{\mathsf{M}_{g}(\FF_{q})}} \sum_{C \in \mathsf{M}_{g}(\FF_{q})} \delta_{(\Dot{m}(C))},
$$
defined on $\Cl G$, converges to $d\dot{m}$ in the weak topology of measures when $q \rightarrow \infty$. Since $$\vaf \circ \h(\theta(C)) = \vaf(\dot{m}(C)),$$
this theorem means also that if $C$ runs over $\mathsf{M}_{g}(\FF_{q})$, the vectors $\theta(C)$ become equidistributed in the fundamental alcove with respect to the Weyl measure when $q \rightarrow \infty$.

\begin{remark}
\label{KSAbelian}
In the preceding theorem, and the above comments, one can substitute the set $\mathsf{A}_{g}(\FF_{q})$ to the set $\mathsf{M}_{g}(\FF_{q})$ \cite[Th. 11.3.10]{KS}. This is an answer to the question raised in the beginning of this section.
\end{remark}

As discussed above, the random variable $\vat(C)$ rules the number of points on the set $\mathsf{M}_{g}(\FF_{q})$, and its law is the counting measure on the closed interval $[- 2g, 2g]$:
$$
\nu_{g,q} = 
\frac{1}{\card{\mathsf{M}_{g}(\FF_{q})}} \sum_{C \in \mathsf{M}_{g}(\FF_{q})} \delta_{(\vat(C))} =
\sum_{x = - 2 g}^{2 g} f_{g,q}(x) \delta_{(x)},
$$
where $\delta_{(x)}$ is the Dirac measure at $x$, with the probability mass function
$$
f_{g,q}(x) = \frac{\card{\set{C \in \mathsf{M}_{g}(\FF_{q})}{(\vat(C) = x}}}{\card{\mathsf{M}_{g}(\FF_{q})}},
$$
defined if $x \in [- 2g, 2g]$ and $q^{1/2}x \in \ZZ$. We put now
$$
\vat(m) = \Trace m, \qquad \vat \circ h(\theta) = 2 \sum_{j = 1}^{g} \cos \theta_{j},
$$
for $m \in G$ and $\theta \in X_{g}$. We take $\vaf(m) = \vat(m)$ in Theorem \ref{KS}, and call $\mu_{\vat}$ be the distribution of the central function $\vat$ as defined at the end of Section \ref{sec_WeilInt}. We obtain:

\begin{corollary}
\label{CorEquiNbPts}
If $q \rightarrow \infty$, the distributions $\nu_{g,q}$  of the Frobenius traces converge to the distribution $\mu_{\vat}$. More precisely, for any continuous function $\varphi$ on $[-2 g, 2 g]$, we have
$$
\oint_{\mathsf{M}_{g}(\FF_{q})} \varphi(\vat(C)) dC =
\int_{- 2 g}^{2 g} \varphi(x) \mu_{\vat}(x) + O\left(q^{-1/2}\right),
$$
and for every $x \in \RR$, we have
$$
\frac{\card{\set{C \in \mathsf{M}_{g}(\FF_{q})}{\vat(C) \leq x}}}{\card{\mathsf{M}_{g}(\FF_{q})}} =
\int_{- \infty}^{x} f_{\vat}(s) ds + O\left(q^{-1/2}\right).
\rlap \qed
$$
\end{corollary}

\begin{lemma}
\label{AnMean}
If $1 \leq n \leq 2 g - 1$,
$$
\oint_{\mathsf{A}_{g}(\FF_{q})}
a_{n}(A) \, dA = \varepsilon_{n} + O\left(q^{-1/2}\right),
$$
where $\varepsilon_{n} = 1$ if $n$ is even and $\varepsilon_{n} = 0$ if $n$ is odd.
\end{lemma}

\begin{proof}
As Equation \eqref{TraceExt} in the appendix, let
$$
\vat_{n}(m) = \text{Trace} (\wedge^{n} \, m),
$$
in such a way that $\vat_{1} = \vat$.
By equality \eqref{IdTraceExt}, we have
$$
a_{n}(A) = \vat_{n} \circ \h(\theta(A)).
$$
Since $\vat_{n} \in \mathscr{T}(G)^{\circ}$, we have, by Remark \ref{KSAbelian},
$$
\oint_{\mathsf{A}_{g}(\FF_{q})}
\vat_{n} \circ \h(\theta(A))  \, dA = \int_{G} \vat_{n}(m) \, dm + O\left(q^{-1/2}\right)
$$
but Lemma \ref{WeightDecomp} implies that the multiplicity of the character $\vat_{0}$ of the unit representation $\mathbf{1}$ is equal to $\varepsilon_{n}$, hence,
$$
\int_{G} \vat_{n}(m) \, dm = \varepsilon_{n}.
\mbox{\qedhere}
$$
\end{proof}

\begin{corollary}
\label{IntVA}
Let $u \in \CC$ and $q \rightarrow \infty$.
\begin{enumerate}
\item
\label{MVA1}
If $\card{u} < q^{1/2}$, then
$$
\oint_{\mathsf{A}_{g}(\FF_{q})} L(A, u) \, dA =
\frac{u^{2 g + 2} - q^{g + 1}}{u^{2} - q} + O\left(q^{g - \frac{1}{2}}\right).
$$
\item
\label{MVA2}
We have
$$
\oint_{\mathsf{A}_{g}(\FF_{q})} \card{A(\FF_{q})} \, dA = q^{g} + O\left(q^{g - 1}\right).
$$
\item
\label{MVA3}
We have
$$
\oint_{\mathsf{M}_{g}(\FF_{q})} \card{C(\FF_{q})}  \, dC = q + O(1).
$$
\end{enumerate}
The implied constants depend only on $g$.
\end{corollary}

\begin{proof}
We have
$$
L(A,u) = q^{g} \bar{L}\left(A,q^{-1/2} u\right) = \sum_{n = 0}^{2 g} (-1)^{n} a_{n}(A) q^{(2 g - n)/2} u^{n},
$$
with $a_{0} = 1$ and $a_{2 g - n} = a_{n}$ for $0 \leq n \leq g$. From Lemma \ref{AnMean}, we get
$$
q^{(2 g - n)/2} u^{n} \oint_{\ \mathsf{F}_{g}(\FF_{q})} a_{n}(A) \, dA =
\varepsilon_{n} q^{(2 g - n)/2} u^{n} +  u^{n} O\left(q^{(2 g - n - 1)/2}\right)
$$
for $1 \leq n \leq 2 g - 1$, and there is no second term in the right hand side if $n = 0$ and $n = 2 g$. Now, if $\card{u} < q^{1/2}$,
$$
\sum_{n = 0}^{2 g} \varepsilon_{n} q^{(2 g - n)/2} u^{n} = \frac{u^{2 g + 2} - q^{g + 1}}{u^{2} - q},
$$
and the absolute value of the difference between this expression and
$$
\oint_{\mathsf{A}_{g}(\FF_{q})} L(A, u) \, dA
$$
is bounded by
$$
B \sum_{n = 1}^{2 g - 1} \card{u}^{n} q^{(2 g - n - 1)/2},
$$
with $B$ depending only on $g$. If $\card{u} \leq q^{1/2}$, then $\card{u}^{n} q^{(2 g - n - 1)/2} \leq q^{(2 g - 1)/2}$, and this proves \eqref{MVA1}. If $\card{u} \leq 1$, then $\card{u}^{n} q^{(2 g - n - 1)/2} \leq q^{g - 1}$, hence,
$$
\oint_{\mathsf{A}_{g}(\FF_{q})} L(A, u) \, dA = q^{g} + O\left(q^{g - 1}\right).
$$
and this proves \eqref{MVA2}, since $\card{A(\FF_{q})} = L(A,1)$. Since Lemma \ref{AnMean} holds by substituting $\mathsf{M}_{g}$ to $\mathsf{A}_{g}$, \eqref{MVA3} is a consequence of this lemma applied to $a_{1}(C)$, and of formula \eqref{Weil}.
\end{proof}

With Corollary \ref{IntVA}\eqref{MVA1}, it appears as though the Frobenius angles were close in the mean to the vertices of the regular polygon with $(2 g + 2)$ vertices, inscribed in the circle of radius $q^{1/2}$, the points $\pm q^{1/2}$ being excluded.

\bigskip
\begin{figure}[ht!]
\centering
\boxed{\includegraphics[scale=0.5]{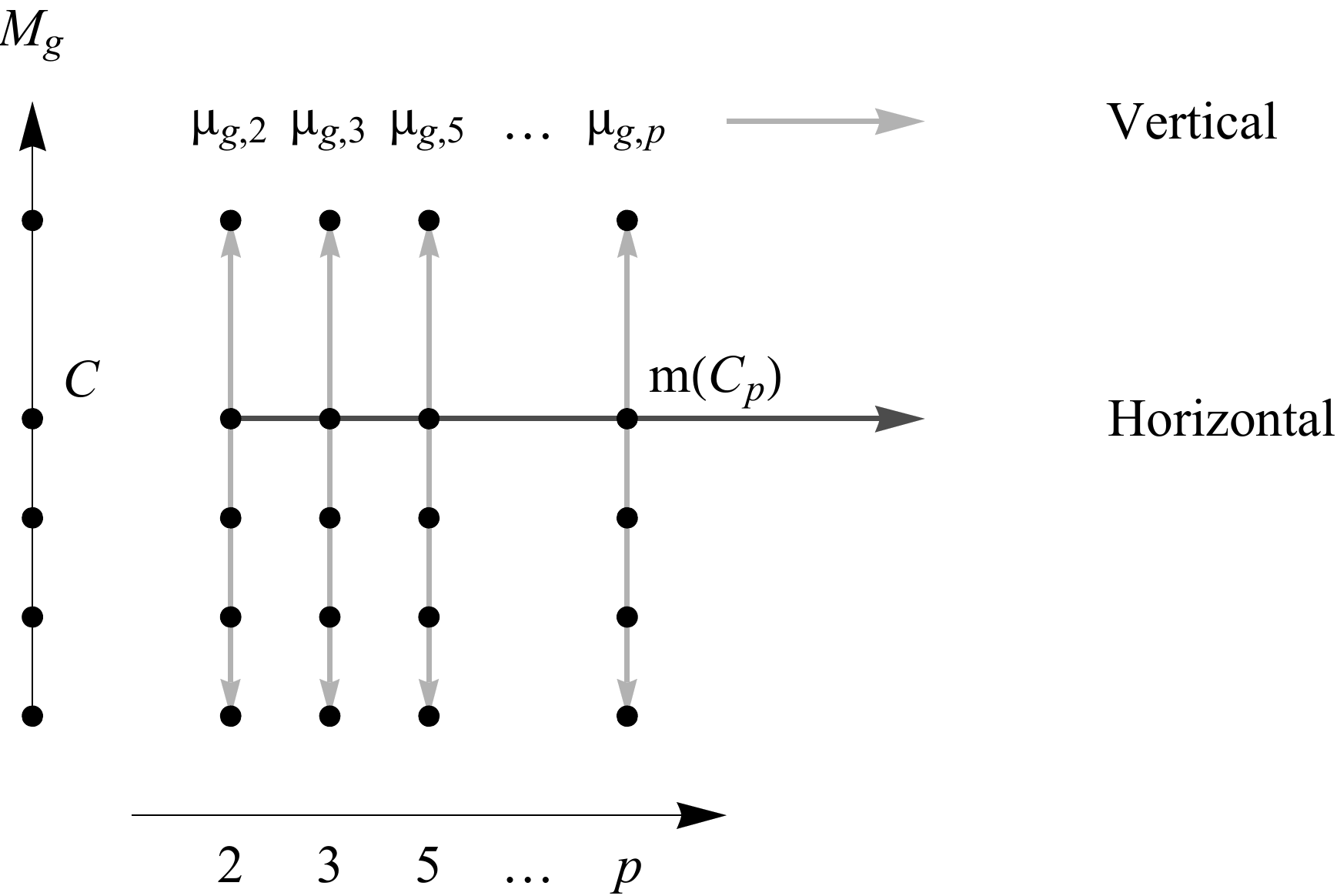}}
\caption{Horizontal versus vertical distribution.}
\label{VertHoriz}
\end{figure}
\bigskip

Another approach on the limiting equidistribution of Frobenius angles is the \emph{generalized Sato-Tate conjecture}, see \cite{Serre2012} for a comprehensive description. Let $C$ be an absolutely irreducible nonsingular projective curve of genus $g$ over $\QQ$, and $S$ a finite subset of prime numbers such that the reduction $C_{p} = C_{\FF_{p}}$ over $\FF_{p}$ is good if $p \notin S$. Then one says that the group $\USp_{2g}$ \emph{arises as the Sato-Tate group} of $C$ if the conjugacy classes $\Dot{m}(C_{p})$ are equidistributed  with respect to the Weyl measure of $G$ when $p \rightarrow \infty$. In other words, this means that if $\vaf \in \VAG$, then
$$
\lim_{n \rightarrow \infty} \oint_{P_{S}(n)} \vaf(\Dot{m}(C_{p})) \, dp =
\int_{G} \vaf(m) \, dm,
$$
where $P_{S}(n) = \set{p \in P \setminus S}{p \leq n}$. The case $g = 1$ is the Sato-Tate original conjecture, now a theorem. Here is an example of what one expects \cite{Ked-Su} :

\begin{conjecture}[Kedlaya-Sutherland]
If $\End_{\CC}(\Jac C) = \ZZ$, and if $g$ is odd, or $g = 2$, or $g = 6$, then the group $\USp_{2g}$ arises as the Sato-Tate group of $C$.
\end{conjecture}

The two preceding types of equidistribution are symbolically shown in Figure \ref{VertHoriz}. The sequence of prime numbers are on the horinzontal axis, and the vertical axis symbolizes the space of curves. The Katz-Sarnak approach is figured as a (horizontal) limit of vertical averages $\mu_{p}$ over vertical lines, and the Sato-Tate approach is a mean performed on horizontal lines.

\section{Expressions of the law of the trace in genus $2$}
\label{sec_expressions}

Assume now $g = 2$. Our purpose is to express the density of the distribution of the trace function $\vat$ on $\USp_{4}$ with the help of special functions. In order to do this, the first step is to compute the characteristic function. The density of the Weyl measure on $X_{2}$ is
$$
\delta_{2}(\theta_{1},\theta_{2}) =
\left(\frac{2}{\pi^{2}}\right) \sin^{2} \theta_{1} \sin^{2} \theta_{2} (2 \cos \theta_{2} - 2 \cos \theta_{1})^{2}.
$$
The fundamental alcove is
$$
A_{2} = \set{(\theta_{1}, \theta_{2}) \in X_{2}}{0 < \theta_{2} < \theta_{1} < \pi}.
$$
The maximum of $\delta_{2}$ in $A_{2}$ is attained at the point
$$
\theta_{m} = (\alpha_{m},\pi - \alpha_{m}), \
\text{where} \ \tan \frac{\alpha_{m}}{2} = \sqrt{2 + \sqrt{3}}, \quad
\delta(\theta_{m}) = \frac{128}{27 \pi^{2}}.
$$
We have $\vat \circ h(\theta_{1},\theta_{2}) = 2 \cos \theta_{1} + 2 \cos \theta_{2}$, and the characteristic function of $\vat$ is
$$
\widehat{f}_{\vat}(t) =  \int_{X_{2}} e^{2 i t (\cos \theta_{1} + \cos \theta_{2})}
\delta_{2}(\theta_{1},\theta_{2}) d\theta_{1} d\theta_{2}.
$$

\begin{proposition}
\label{Bessel}
For every $t \in \RR$, we have
$$
\widehat{f}_{\vat}(t) = \frac{4 J_{1}(2 t)^{2}}{t^{2}}
- \frac{6 J_{1}(2 t) J_{2}(2 t)}{t^{3}}
+  \frac{4 J_{2}(2 t)^{2}}{t^{2}}.
$$
\end{proposition}

Here, $J_{1}$ and $J_{2}$ are \emph{Bessel functions}.

\begin{proof}
Let
$$
\begin{array}{llllll}
V_{a}(x) & = & 2^{5} \cos^{2}x \sin^{2}x & = & 8 \sin^{2} (2x) \\
V_{b}(x) & = & 2^{5} \cos^{2}x \sin^{2}x \cos 2x & = & 4 \sin 2 x \cos 4 x \\
V_{c}(x) & = & 2^{5} \cos^{2}x \sin^{2}x \cos^{2} 2x & = & 2 \sin^{2} 4 x \\
\end{array}
$$
Then
$$
32 \pi^{2} \delta(x,y) = V_{c}(x) V_{a}(y) + V_{a}(x) V_{c}(y) - 2 V_{b}(x) V_{b}(y).
$$
and
$$
\widehat{F}(t) = 2 \widehat{V}_{a}(t) \widehat{V_{c}}(t) - 2 \widehat{V_{b}}(t)^{2}.
$$
But
$$
\widehat{V}_{a}(t) = \dfrac{  \sqrt{2}}{t} \ J_{1}(2 t), \quad
\widehat{V_{b}}(t) = \dfrac{i \sqrt{2}}{t} \ J_{2}(2 t), \quad
\widehat{V_{c}}(t) = \dfrac{  \sqrt{2}}{t} \ J_{1}(2 t) - \dfrac{3}{\sqrt{2}t^{2}} J_{2}(2 t),
$$
and the result follows.
\end{proof}

We now compute the moments $M_{n}(\vat)$ of $\vat$. By Proposition \ref{Bessel}, the characteristic function can be expressed by a \emph{generalized hypergeometric series} \cite[\S 9.14, p. 1010]{GR}:
$$
\widehat{f}_{\vat}(t)
= {_{1}F_{2}}\left( \frac{3}{2} ; 3, 4; - 4 t^{2} \right)
= \sum_{n = 0}^{\infty} (-1)^{n}
\frac{(\frac{3}{2})_{n}}{(3)_{n} (4)_{n}} \,  2^{2 n} \, \frac{t^{2 n}}{n!},
$$
where $(a)_{n} = a (a + 1) \dots (a + n - 1)$ is the Pochhammer's symbol. It then follows from \eqref{CaracMoments} that the odd moments are equal to zero. Since
$$
\frac{(\tfrac{3}{2})_{n}}{(3)_{n} (4)_{n}} =
\frac{24}{\sqrt{\pi}}\frac{(n + \frac{1}{2})\Gamma(n + \frac{1}{2})}{\Gamma(n + 3)\Gamma(n + 4)}
$$
and \cite[p. 897]{GR}
$$
\Gamma\left(n + \frac{1}{2}\right) = \sqrt{\pi} \, 2^{- 2n} \, \frac{(2 n)!}{n!},
$$
we obtain
$$
M_{2 n}(\vat) = \frac{6.(2n)!(2 n + 2)!}{n!(n + 1)!(n + 2)!(n + 3)!} \quad \text{for} \ n \geq 0.
$$
One finds as expected Mihailovs' formula, in accordance with \cite[\S 4.1]{Ked-Su}, which includes another formula for $\widehat{f}_{\vat}(t)$, and also \cite[p. 126]{Serre2012}.

In what follows, four different but equivalent expressions for the distribution of $\tau$ are given.

\subsection{Hypergeometric series}
An expression of the density $f_{\vat}$ of the distribution of $\vat$ is the following. Recall that \emph{Gauss' hypergeometric series}
$$
{_{2}F_{1}}(a,b;c;z) = \sum_{n = 0}^{\infty} \frac{(a)_{n}(b)_{n}}{(c)_{n}} \, \frac{z^{n}}{n!}
$$
converges if $\card{z} < 1$ \cite[\S 9.1.0, p. 1005]{GR}.

\begin{theorem}
\label{hypergeometric}
If $\card{x} < 4$, we have
$$
f_{\vat}(x) = \frac{1}{4 \pi} \left(1 - \frac{x^{2}}{16}\right)^{4} \ {_{2}F_{1}}\left(\frac{3}{2}, \frac{5}{2} ; 5 ; 1 - \frac{x^{2}}{16} \right).
$$
\end{theorem}

This theorem immediately follows from the following lemma.

\begin{lemma}
\label{Distrib1}
If $\card{x} < 4$, we have
$$
f_{\vat}(x) = \frac{64}{5 \pi^{2}} m(x)^{4} I(m(x)), \quad \text{where} \quad m(x) = 1 - \frac{x^{2}}{16},
$$
and
$$
I(m) = \int_{0}^{1} t^{2} \left(\frac{1 - t^{2}}{1 - m t^{2}}\right)^{\frac{5}{2}} dt.
$$
Moreover
$$
I(m) = \frac{5 \pi}{256} \ {_{2}F_{1}}\left(\frac{3}{2}, \frac{5}{2} ; 5 ; m \right).
$$
\end{lemma}

\begin{proof}
We use a \emph{formula of Schl\"afli}, see \cite[Eq. 1, p. 150]{Watson}. If  $\mu$ and $\nu$ are real numbers, then
$$
J_{\mu}(t) J_{\nu}(t)
= \frac{2}{\pi} \int_{0}^{\pi/2}
J_{\mu + \nu}(2 t \cos \varphi) \cos(\mu - \nu) \varphi\, d\varphi \quad (\mu + \nu > - 1).
$$
As particular cases of this formula, we get
\begin{align*}
J_{1}(t)^{2}
& = \frac{2}{\pi} \int_{0}^{4}
J_{2}\left(\frac{u t}{2}\right) \frac{d u}{\sqrt{16 - u^{2}}} \, \\
J_{1}(t) J_{2}(t)
& = \frac{2}{\pi} \int_{0}^{4}
J_{3}\left(\frac{u t}{2}\right) \frac{u}{4} \, \frac{d u}{\sqrt{16 - u^{2}}} \, ,\\
J_{2}(t)^{2} 
& = \frac{2}{\pi} \int_{0}^{4}
J_{4}\left(\frac{u t}{2}\right)  \, \frac{d u}{\sqrt{16 - u^{2}}} \, .
\end{align*}
By transferring these equalities in Proposition \ref{Bessel}, we obtain
\begin{multline*}
\widehat{f_{\vat}}(t) 
= \frac{4}{t^{2}} J_{1}(2 t)^{2} - \frac{6}{t^{3}} J_{1}(2 t) J_{2}(2 t) +
\frac{4}{t^{2}} J_{2}(2 t)^{2} \\
= \frac{2}{\pi} \int_{0}^{4}
\left[\frac{4}{t^{2}} J_{2}(u t) - \frac{3 u}{2 t^{3}} J_{3}(u t)
+ \frac{4}{t^{2}} J_{4}(u t) \right] \frac{d u}{\sqrt{16 - u^{2}}} \, .
\end{multline*}
and since
$$
f_{\vat}(x) = \frac{1}{\pi} \int_{0}^{\infty} \widehat{f_{\vat}}(t)  \cos t x \, dt,
$$
we have
\begin{equation}
\label{IntDouble}
f_{\vat}(x) =  \frac{2}{\pi^{2}} \int_{0}^{4} \frac{d u}{\sqrt{16 - u^{2}}}
\int_{0}^{\infty}
\left[\frac{4}{t^{2}} J_{2}(u t) - \frac{3 u}{2 t^{3}} J_{3}(u t)
+ \frac{4}{t^{2}} J_{4}(u t) \right] \cos t x \, dt.
\end{equation}
We use now a \emph{formula of Gegenbauer} on the cosine transform, see \cite[p. 409]{MOS} and \cite[Eq. 3, p. 50]{Watson}. Assume $\Real \nu > - 1/2$, $u > 0$ and let $n$ be an integer $\geq 0$. If $0 < x < u$, then
$$
\int_{0}^{\infty} t^{-\nu} J_{\nu + 2 n}(u t) \cos t x \, dt =
(-1)^{n} 2^{\nu - 1} u^{- \nu}\frac{\Gamma(\nu)}{\Gamma(2 \nu + n)}  (u^{2} - x^{2})^{\nu - 1/2}
C_{2 n}^{\nu}\left(\frac{x}{u}\right),
$$
where $C_{n}^{\nu}(x)$ is the \emph{Gegenbauer polynomial}. If $u < x < \infty$, this integral is equal to $0$. From Gegenbauer's formula we deduce that if $0 < x < u$, then
$$
\int_{0}^{\infty} t^{-2} J_{2}(u t) \cos t x \, dt =
\frac{1}{3} \frac{(u^{2} - x^{2})^{3/2}}{u^{2}},
$$
$$
\int_{0}^{\infty} t^{-3} J_{3}(u t) \cos t x \, dt =
\frac{1}{15} \frac{(u^{2} - x^{2})^{5/2}}{u^{3}},
$$
$$
\int_{0}^{\infty} t^{-2} J_{4}(u t) \cos t x \, dt =
- \frac{1}{30} \frac{(u^{2} - x^{2})^{3/2}}{u^{2}} \left(\frac{12 x^{2}}{u^{2}} - 2\right),
$$
since $C_{2}^{2}(x) = 12 x^{2} - 2$. Transferring these relations in \eqref{IntDouble}, we get
\begin{multline*}
5 \pi^{2} f_{\vat}(x)  = 
16 \int_{x}^{4}
  \frac{(u^{2} - x^{2})^{3/2}}{u^{2}} \, \frac{d u}{\sqrt{16 - u^{2}}} \\
- \int_{x}^{4}
  \frac{(u^{2} - x^{2})^{5/2}}{u^{2}} \, \frac{d u}{\sqrt{16 - u^{2}}} 
- 16 x^{2} \int_{x}^{4}
  \frac{(u^{2} - x^{2})^{3/2}}{u^{4}} \, \frac{d u}{\sqrt{16 - u^{2}}} \, ,
\end{multline*}
and this leads to
$$
f_{\vat}(x) = \frac{1}{5 \pi ^{2}}
\int_{x}^{4}
\frac{(u^{2} - x^{2})^{5/2}}{u^{4}} \sqrt{16 - u^{2}} \, du.
$$
By the change of variables
$$
u = 4 \sqrt{1 - m(x) t^{2}}, \quad \text{where} \quad m(x) = 1 - \frac{x^{2}}{16}.
$$
we obtain
$$
f_{\vat}(x) = \frac{64 \, m(x)^{4}}{5 \pi ^{2}} \int_{0}^{1}
t^{2} \left(\frac{1 - t^{2}}{1 - m(x) t^{2}}\right)^{\frac{5}{2}} dt,
$$
which is the first result. Euler's integral representation of the hypergeometric series is
$$
{_{2}F_{1}}(a,b;c;z) = \frac{\Gamma(c)}{\Gamma(b)\Gamma(c - b)}
\int_{0}^{1}\frac{t^{b - 1} (1 - t)^{c - b - 1}}{(1 - t z)^{a}} \, dt
$$
if $\Real z >0$, and $\Real c > \Real b > 0$. From this we deduce, with the change of variables $t = u^{2}$, that
$$
I(m) = \frac{5 \pi}{256} \ {_{2}F_{1}}\left(\frac{3}{2}, \frac{5}{2} ; 5 ; m\right),
$$
which is the second result. 
\end{proof}

\subsection{Legendre function}

Another expression of $f_{\vat}$ is given by the \emph{associated Legendre function of the first kind} $\Legendre{b}{a}(z)$, defined in the half-plane $\Real z > 1$. If $a$ is not an integer $\geq 1$, and if $b > 3/2$, this function is defined by \cite[Eq. 8.702, p. 959]{GR} :
$$
\Legendre{b}{a}(z) =
\frac{1}{\Gamma(1 - a)} \left(\frac{z + 1}{z - 1}\right)^{\frac{a}{2}}
{_{2}F_{1}}\left(- b, b + 1; 1 - a ; \frac{1 - z}{2}\right).
$$
If $a = m$ is an integer  and if $z > 1$ is real, then \cite[Eq. 8.711.2, p. 960]{GR} :
$$
\Legendre{b}{m}(z) = \frac{(b + 1)_{a}}{\pi}
\int_{0}^{\pi} \left(z + \sqrt{z^{2} - 1} \cos \varphi\right)^{b} \cos m\varphi \, d\varphi.
$$
If $a = 0$, this is the \emph{Laplace integral}.

\begin{theorem}
\label{Legendre}
If $\card{x} < 4$, then
$$
f_{\vat}(x) = - \frac{64}{15 \pi} \, \sqrt{\card{x}} \left(1-\frac{x^2}{16}\right)^2
\Legendre{\frac{1}{2}}{2}\left(\frac{x^{2} + 16}{4 x}\right).
$$
\end{theorem}

\begin{proof}
By Theorem \ref{hypergeometric}, we have
$$
F(x) = \frac{1}{4 \pi} m(x)^{4} \ {_{2}F_{1}}\left(\frac{3}{2}, \frac{5}{2} ; 5 ; m(x)\right).
$$
But \cite[p. 51]{MOS}
$$
{_{2}F_{1}}\left(\frac{3}{2}, \frac{5}{2} ; 5 ; z\right) =
(1 - z)^{-3/4}
{_{2}F_{1}}\left(\frac{3}{2}, \frac{7}{2} ; 3 ; - \frac{(1 - \sqrt{1 - z})^{2}}{4 \sqrt{1 - z}}\right)
$$
and \cite[p. 47]{MOS}
$$
{_{2}F_{1}}\left(\frac{3}{2}, \frac{7}{2} ; 3 ; z\right) =
(1 - z)^{-2}
{_{2}F_{1}}\left(-\frac{1}{2}, \frac{3}{2} ; 3 ; z\right).
$$
On the other hand, if $z = m(x)$, then
$$
- \frac{(1 - \sqrt{1 - z})^{2}}{4 \sqrt{1 - z}} = - \frac{(x - 4)^{2}}{16 x} \, .
$$
By the definition of Legendre functions,
$$
\Legendre{\frac{1}{2}}{-2}\left(\frac{1}{2}\left(\frac{x}{4} + \frac{4}{x}\right)\right) = 
\left(\frac{x - 4}{x + 4}\right)^{4} \
{_{2}F_{1}}(-\frac{1}{2}, \frac{3}{2} ; 3 ; - \frac{(x - 4)^{2}}{16 x}),
$$
and this implies
$$
f_{\vat}(x) = \frac{4 }{\pi} \, \sqrt{x} \, \left(1-\frac{x^2}{16}\right)^2
\Legendre{\frac{1}{2}}{-2}\left(\frac{1}{2}\left(\frac{x}{4} + \frac{4}{x}\right)\right).
$$
Since
$$
\Legendre{b}{m}(z) =
\frac{\Gamma(b + m + 1)}{\Gamma(b - m + 1)} \, \Legendre{b}{-m}(z)
$$
if $m \in \ZZ$, we obtain the required expression.
\end{proof}

Since ${_{2}F_{1}}(a,b;c;0) = 1$, we deduce from Theorem \ref{hypergeometric} that
$$
f_{\vat}(x) = \frac{1}{4 \pi} \left(1 - \frac{x^{2}}{16}\right)^{4} + O(x - 4)^{5}
$$
and hence, in accordance with \cite[p. 126]{Serre2012}:

\begin{corollary}
If $\card{x} = 4 - \varepsilon$, with $\varepsilon \rightarrow 0$ and $\varepsilon > 0$, then
$$
f_{\vat}(x) = \frac{\varepsilon^{4}}{6 4 \pi} + O(\varepsilon^{5}). \rlap \qed
$$
\end{corollary}

Since
$$
\lim_{x \rightarrow 0} \sqrt{x} \,
\Legendre{\frac{1}{2}}{2}\left(\frac{1}{2}(\frac{x}{4} + \frac{4}{x})\right)
= - \frac{1}{\pi} \, ,
$$
we deduce from Proposition \ref{Legendre} that the maximum of $f_{\vat}$ is reached for $x = 0$, and
$$
f_{\vat}(0) = \frac{64}{15 \pi^{2}}  = 0.432 \dots
$$
The graph of $f_{\vat}$ is given in Figure \ref{curve} ; we recover the curve drawn in \cite[p. 124]{Ked-Su}.

\bigskip
\begin{figure}[ht!]
\centering
\boxed{\includegraphics[scale=0.5]{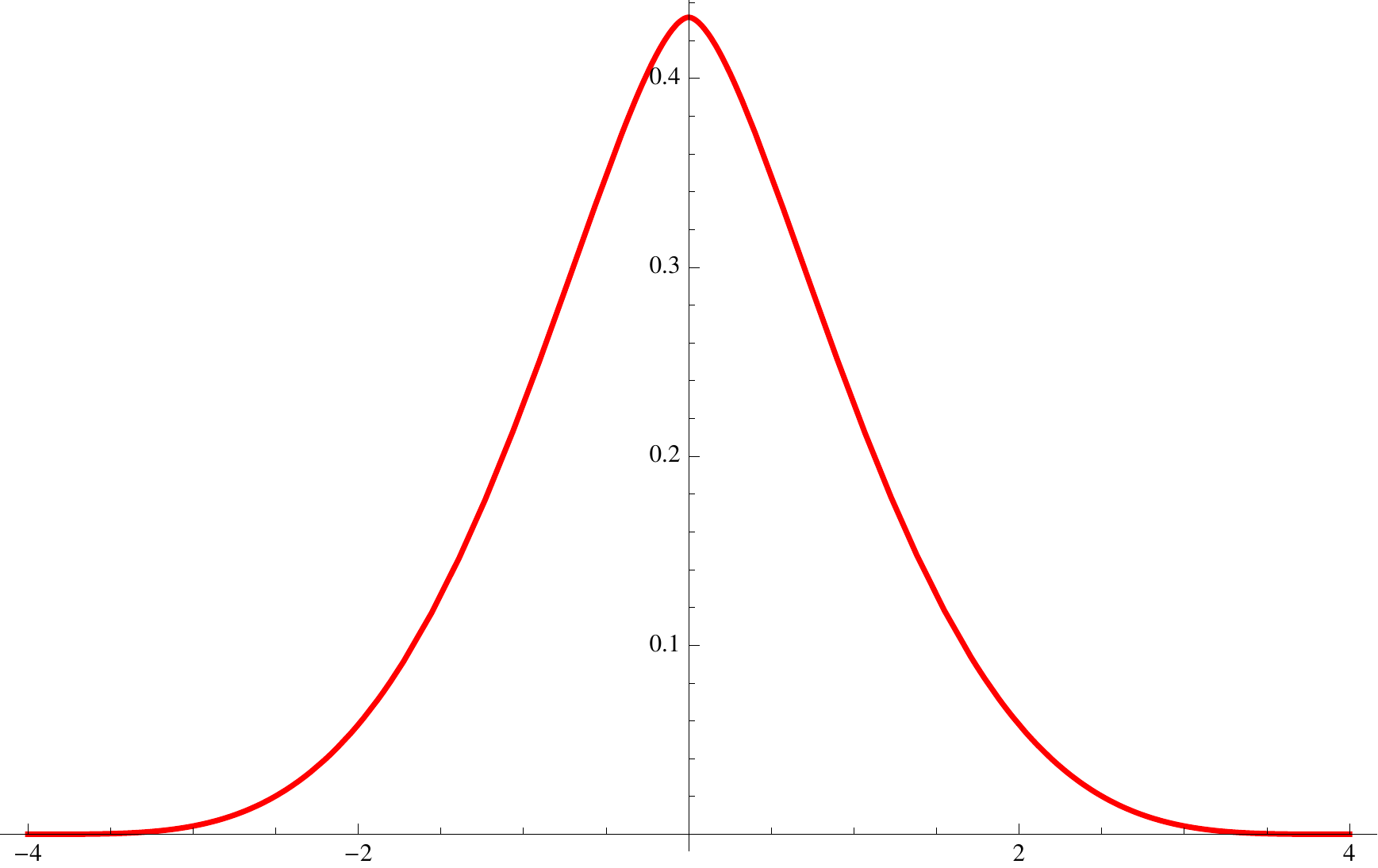}}
\caption{Density of the distribution of $\vat$, case $g = 2$.}
\label{curve}
\end{figure}
\bigskip

\subsection{Elliptic integrals}

Another expression of $f_{\vat}$ is given by \emph{Legendre elliptic integrals}. Let
$$
K(m) = \int_{0}^{\pi/2} \frac{d\varphi}{\sqrt{1 - m \sin^{2} \varphi}}, \quad
E(m) = \int_{0}^{\pi/2} \sqrt{1 - m \sin^{2} \varphi} \ d\varphi,
$$
be the Legendre elliptic integrals of first and second kind, respectively. The implementation of $f_{\vat}$ in the \emph{Maple} software gives:

\begin{corollary}
If $\card{x} < 4$, then
$$
\frac{15}{64} \pi ^{2} f_{\vat}(x) =
(m^{2} - 16 m + 16) E(m) - 8(m^{2} - 3 m + 2) K(m),
$$
where $m = 1 - (x^2/16)$. \qed
\end{corollary}
The mention of the existence of such a formula is made in \cite{FKRS}.

\subsection{Meijer $G$-functions}

Another expression of $f_{\vat}$ is given by \emph{Meijer $G$-functions} \cite[\S 9.3, p. 1032]{GR}. They are defined as follows : take $z$ in $\CC$ with $0 < \card{z} < 1$ and $m,n,p,q$ in $\NN$. Then
\begin{multline*}
\Meijer{m}{n}{p}{q}{z}
{\begin{array}{c} a_{1}, \dots, a_{p} \\ b_{1}, \dots, b_{q}
\end{array}}
\\ =
\frac{1}{2 i \pi} \int_{C}
\frac{
\prod_{k = 1}^m     \Gamma(s + b_{k})}{
\prod_{k = n + 1}^p \Gamma(s + a_{k})}
. \frac{
\prod_{k = 1}^n     \Gamma(- s - a_{k} + 1)}{
\prod_{k = m + 1}^q \Gamma(- s - b_{k} + 1)} \, z^{- s} ds
\end{multline*}
Here, $a_{1}, \dots, a_{p}, b_{1}, \dots, b_{q}$ are \textit{a priori} in $\CC$, and $C$ is a suitable Mellin-Barnes contour.

\begin{corollary}
\label{MeijerLaw}
If $\card{x} < 4$,
$$
f_{\vat}(x) = \frac{6}{\pi} G\left(\frac{x^{2}}{16}\right), \quad \text{with} \quad
G(z) = \Meijer{2}{0}{2}{2}{z}
{\begin{matrix} \frac{5}{2}, & \frac{7}{2} \\ 0, & 1 \end{matrix}}.
$$
We have
$$
G(z) = \frac{1}{2 i \pi} \int_{\Real s = c}
\frac{\Gamma(s) \Gamma(s + 1)}
{\Gamma\left(s + \frac{5}{2}\right) \Gamma\left(s + \frac{7}{2}\right)} \, z^{- s} ds,
$$
with $0 < c < \frac{1}{2}$.
\end{corollary}

\begin{proof}
If $\card{z} < 1$, then \cite[07.34.03.0653.01]{Wolfram}:
\begin{multline*}
G_{2,2}^{2,0}\left(z\left|
\begin{array}{c}
 a,c \\
 b,-a+b+c \\
\end{array}
\right.\right) \\ =
\frac{\sqrt{\pi}}{\Gamma (a-b)} (1-z)^{a-b-\frac{1}{2}} z^{\frac{1}{4} (-2 a+2 c-1)+b}
\mathfrak{P}_{-a+c-\frac{1}{2}}^{-a+b+\frac{1}{2}}\left(\frac{z+1}{2 \sqrt{z}}\right)
\end{multline*}
and the left hand side is equal to zero if $\card{z} > 1$. Hence, if $\card{z} < 1$,
$$
\Meijer{2}{0}{2}{2}{z}
{\begin{matrix} \frac{5}{2}, & \frac{7}{2}\\ 0, & 1 \end{matrix}} =
\frac{4}{3} (1 - z)^{2} z^{1/4} \ \mathfrak{P}_{1/2}^{-2}\left(\frac{z + 1}{2 \sqrt{z}}\right),
$$
and we apply Theorem \ref{Legendre}.
\end{proof}

\begin{corollary}
If $\card{x} < 4$, then the repartition function of $\vat$ is
$$
\Phi_{\vat}(x) = \frac{3 x}{\pi}
G\left(\frac{x^{2}}{16}\right)
+ \frac{1}{2} \, ,
$$
with
$$
G(z) = 
\Meijer{2}{1}{3}{3}{z}
{\begin{matrix} \frac{1}{2}, & \frac{5}{2}, & \frac{7}{2} \\ 0, & 1, &- \frac{1}{2} \end{matrix}} .
$$
\end{corollary}

\begin{proof}
According to \cite[07.34.21.0003.01]{Wolfram}, we have
$$
\int z^{\alpha - 1} \,
\Meijer{m}{n}{p}{q}{z}
{\begin{array}{c} a_{1}, \dots, a_{p} \\ b_{1}, \dots, b_{q}\end{array}} dz = 
z^{\alpha} \,
\Meijer{m}{n + 1}{p + 1}{q + 1}{z} 
{\begin{array}{c} 1 - \alpha, a_{1}, \dots, a_{p} \\ b_{1}, \dots, b_{q}, - \alpha \end{array}}.
$$
By Corollary \ref{MeijerLaw}, a primitive of $f_{\vat}$ is
$$
\Phi_{0}(x) = \frac{6}{\pi} \int
\Meijer{2}{0}{2}{2}{\frac{x^{2}}{16}}
{\begin{matrix} \frac{5}{2}, & \frac{7}{2} \\ 0, & 1 \end{matrix}} =
\frac{3 x}{\pi} 
\Meijer{2}{1}{3}{3}{\frac{x^{2}}{16}}
{\begin{matrix} \frac{1}{2}, & \frac{5}{2}, & \frac{7}{2} \\ 0, & 1, &- \frac{1}{2} \end{matrix}},
$$
and $\Phi_{0}(-4) = -1/2$.
\end{proof}

\bigskip
\begin{figure}[ht!]
\centering
\boxed{\includegraphics[scale=0.4]{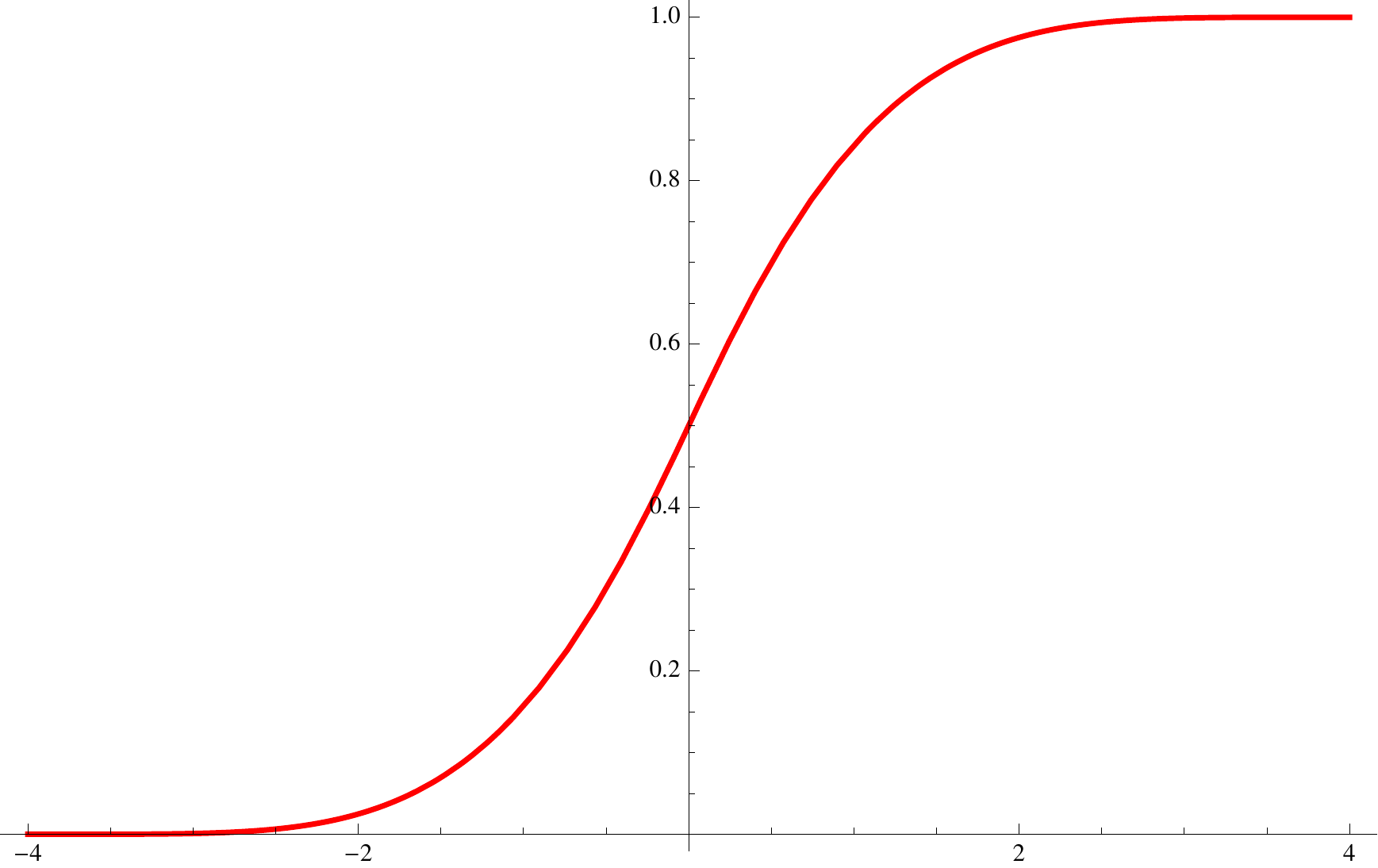}}
\caption{Repartition function of $\vat$}
\label{repartition}
\end{figure}
\bigskip

\subsection{The trace in $\SU_{2} \times \SU_{2}$}

In order to present a comparison with the above results, we give here without proof the distribution of the trace of a compact semi-simple subgroup of rank $2$ of $\USp_{4}$, namely, the group $\SU_{2} \times \SU_{2}$. If
$$
x = (x_{1}, x_{2}) \quad \text{and} \quad x_{i} =
\begin{pmatrix} a_{i} & - \bar{b}_{i} \\ b_{i} & \bar{a}_{i} \end{pmatrix} \in \SU_{2}, \quad
\card{a_{i}}^{2} + \card{b_{i}}^{2} = 1,
\quad i = 1,2,
$$
the map
$$
\pi(x) =
\begin{pmatrix}
a_{1} & 0     & - \bar{b}_{1} & 0 \\
0     & a_{2} & 0             & - \bar{b}_{2} \\
b_{1} & 0     & \bar{a}_{1}   & 0 \\
0     & b_{2} & 0             & \bar{a}_{2}  
\end{pmatrix}
$$
defines an embedding
$$
\begin{CD}
\pi : \SU_{2} \times \SU_{2} & @>>> & \USp_{4}
\end{CD}
$$
whose image contains the maximal torus $T$ of $\USp_{4}$. We put
$$
\boldsymbol{\rho}(x) = \Trace \pi(x).
$$
The characteristic function of $\boldsymbol{\rho}$ is the square of the characteristic function of the distribution of the trace of $\SU_{2}$:
$$
\widehat{f_{\boldsymbol{\rho}}}(t) = \frac{J_{1}(2 t)^{2}}{t^{2}}.
$$
The even moments are equal to zero, and the odd moments are
$$
M_{2 n}(\boldsymbol{\rho}) = C_{n} C_{n + 1} = \frac{2 (2 n)! (2 n + 1)!}{(n!)^{2} (n + 1)! (n + 2)!} \, .
$$
where
$$
C_{n} = \frac{1}{n + 1}\binom{2 n}{n}
$$
is the $n$th \emph{Catalan number}. One finds the sequence
$$
1, 0, 2, 0, 10, 0, 70, 0, 588, 0, 5544 \dots
$$
in accordance with the sequence A005568 in the OEIS \cite{Sloane}.

\begin{theorem}
\label{DistribG}
If $\card{x} < 4$, the density of the distribution of $\boldsymbol{\rho}$ is
$$
f_{\boldsymbol{\rho}}(x) = \frac{1}{2 \pi} \left(1 - \frac{x^{2}}{16}\right)^{2} \ {_{2}F_{1}}\left(\frac{1}{2}, \frac{3}{2} ; 3 ; 1 - \frac{x^{2}}{16} \right).
\rlap \qed
$$
\end{theorem}

\begin{corollary}
If $\card{x} = 4 - \varepsilon$, with $\varepsilon \rightarrow 0$ and $\varepsilon > 0$, then
$$
f_{\boldsymbol{\rho}}(x) = \frac{\varepsilon^{2}}{8 \pi} - \frac{\varepsilon^{3}}{64 \pi} -
\frac{\varepsilon^{4}}{4096 \pi}+ O(\varepsilon^{5}).
\rlap \qed
$$
\end{corollary}

The maximum of $f_{\boldsymbol{\rho}}$ is reached for $x = 0$, and
$$
f_{\boldsymbol{\rho}}(0) = \frac{8}{3 \pi^{2}} = 0.270\dots
$$
The graph of $f_{\boldsymbol{\rho}}$ is given in Figure \ref{DessinLoiT}.

\begin{figure}[htbp]
\boxed{\includegraphics[scale=0.7]{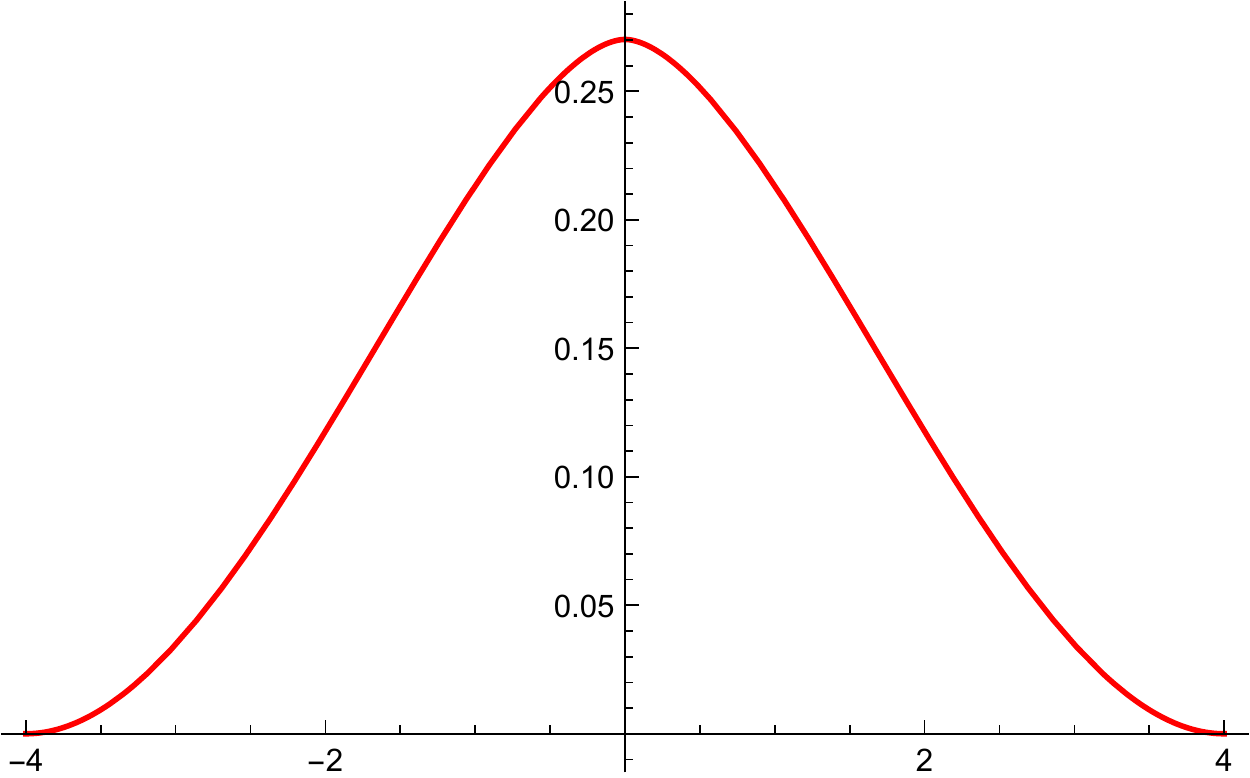}}
\caption{Density of the distribution of $\boldsymbol{\rho}$.}
\label{DessinLoiT}
\end{figure}

\section{The Vi\`ete map and its image}
\label{sec_VieteMap}

Another approach of the distribution of the trace is realized by an algebraic form of Weyl's integration formula, using symmetric polynomials. This comes from a general program developed by Kohel \cite{Kohel}, formerly outlined by DiPippo and Howe in \cite{Howe}. If $t = (t_{1}, \dots, t_{g}) \in \CC^{g}$, consider a monic polynomial
\begin{equation}
\label{roots}
h_{t}(u) = (u - t_{1}) \dots (u - t_{g}) = u^{g} - s_{1}(t) u^{g - 1} + \dots + (-1)^{g} s_{g}(t)
\end{equation}
in $\CC[u]$. Here
$$s_{n}(t) = \sum_{i_{1} < \dots < i_{n}} t_{i_{1}} \dots t_{i_{k}}$$
is the \emph{elementary symmetric polynomial} of degree $n$ in $g$ variables. The discriminant of $h_{t}$ is
\begin{equation}
\label{D0disc}
\disc h_{t} = D_{0}(t) = \prod_{j < k}(t_{k} - t_{j})^{2}.
\end{equation}
The \emph{Vi\`ete map} $\mathsf{s}: \CC^{g} \longrightarrow \CC^{g}$ is the surjective polynomial mapping
$$
\mathsf{s}(t_{1}, \dots, t_{g}) = (s_{1}(t), \dots, s_{g}(t)),
$$
where $t = (t_{1}, \dots, t_{g})$, inducing a bijection
$$
\begin{CD}
\CC^{g}/\mathfrak{S}_{g} & @>{\sim}>> & \CC^{g}
\end{CD}
$$
which is a homeomorphism, because the map between the corresponding projective spaces is a continuous bijection between compact spaces. Hence, the Vi\`ete map is open and proper. We denote by
$$
\Pi_{g} = \mathsf{s}(\RR^{g})
$$
the closed subset which is the image of the Vi\`ete map. Hence, $(s_{1}, \dots, s_{g}) \in \Pi_{g}$ if and only $h_{t}(u)$ has only real roots. The induced map
$$
\begin{CD}
\RR^{g}/\mathfrak{S}_{g} & @>{\sim}>> & \Pi_{g}
\end{CD}
$$
is a homeomorphism. The \emph{fundamental chamber} of $\RR^{g}$ related to $\mathfrak{S}_{g}$ is
$$
C_{g} = \set{t \in \RR^{g}}{t_{1} < t_{2} < \dots < t_{g}}
$$
and $\bar{C}_{g}$ is a fundamental domain for $\mathfrak{S}_{g}$ in $\RR^{g}$. We are going to show that $\mathsf{s}$ is a local diffeomorphism at the points of an open dense subset of $\RR^{g}$. For this purpose, we calculate $J(\mathsf{s})$, where $J(\mathsf{f})$ denotes the jacobian matrix of a polynomial map $\mathsf{f}: \CC^{g} \longrightarrow \CC^{g}$. Recall that the power sums
$$
p_{n}(t) = t_{1}^{n} + \dots + t_{g}^{n} \quad  (n \geq 1)
$$
can be expressed in terms of elementary symmetric polynomials. Precisely, from Newton's relations
$$
p_{n} = \sum_{j = 1}^{n - 1} (- 1)^{j - 1} s_{j} p_{n - j} + (- 1)^{n - 1} n s_{n} \quad (n \geq 1),
$$
we obtain \cite[p. 28]{Macdonald} :
$$
p_{n} = \begin{vmatrix}
s_{1}   & 1         & 0         & \dots & 0 \\
2 s_{2} & s_{1}     & 1         & \dots & 0 \\
\dots   & \dots     & \dots     & \dots & \dots \\
n s_{n} & s_{n - 1} & s_{n - 2} & \dots & s_{1}
\end{vmatrix}.
$$
This is related to a more suitable expression \cite[p. 72]{Dickson}, \cite[Ch. IV, \S{} 6, Ex. 6]{BkiALG47}, obtained by Albert Girard \cite{Girard} in 1629, and sometimes attributed to Waring (1762):

\begin{proposition}[Girard's formula]
\label{Albert}
If $1 \leq n \leq g$ and $s = (s_{1}, \dots, s_{g})$, let
$$
v_{n}(s) =  n \sum_{b \in \EuScript{P}_{n}}
\, \frac{(b_{1} + b_{2} + \dots + b_{g} - 1)!}{b_{1}! \dots b_{g}!} \,
u_{1}^{b_{1}} \dots u_{g}^{b_{g}},
$$
where $u_{n} = (- 1)^{n - 1} s_{n}$ for $1 \leq n \leq g$, and the summation being extended to
$$
\EuScript{P}_{n} = \set{b = (b_{1}, \dots, b_{g}) \in \NN^{g}}
{b_{1} + 2 b_{2} + \dots + g b_{g} = n}.
$$
Then
$$p_{n} = v_{n} \circ \mathsf{s}. \rlap \qed$$
\end{proposition}

The map $\varphi \mapsto \varphi \circ \mathsf{s}$ defines an isomorphism
$$
\begin{CD}
\mathsf{s}^{*} : \ZZ[s_{1}, \dots, s_{g}] & @>{\sim}>> & \ZZ[t_{1}, \dots, t_{g}]^{\sym}.
\end{CD}
$$
Since $D_{0} \in \ZZ[t_{1}, \dots, t_{g}]^{\sym}$, there is a polynomial $d_{0} \in \ZZ[s_{1}, \dots, s_{g}]$ such that
\begin{equation}
\label{d0sym}
d_{0}(\mathsf{s}(t)) = D_{0}(t) = \prod_{j < k} (t_{k} - t_{j})^{2}.
\end{equation}
Let
$$
U_{g} = \set{t \in \RR^{g}}{D_{0}(t) \neq 0}, \quad
\Pi_{g}^{\circ} = \set{s \in \RR^{g}}{d_{0}(s) \neq 0}.
$$
Then $\Pi_{g}^{\circ} = \mathsf{s}(U_{g})$, and $\Pi_{g}^{\circ}$ is a dense open set of $\Pi_{g}$. The roots of the polynomial $h_{t} \in \RR[u]$ given by \eqref{roots} are real and simple if and only if $\mathsf{s}(t) \in \Pi^{\circ}_{g}$. 

\begin{proposition}
\label{IsoViete}
With the preceding notation:
\label{change}
\begin{enumerate}
\item
\label{CHG1}
If $t \in \RR^{g}$, then
$$\card{\det J(\mathsf{s})(t)} = \sqrt{D_{0}(t)} = \prod_{j < k}\card{t_{k} - t_{j}}.$$
\item
\label{CHG2}
The map $\mathsf{s}$ is a local diffeomorphism at every point of $U_{g}$.
\item
\label{CHG3}
The map $\mathsf{s}$ is a diffeomorphism from the fundamental chamber $C_{g}$ to $\Pi_{g}^{\circ}$.
\end{enumerate}
\end{proposition}

\begin{proof}
Define two polynomial maps from $\CC^{g}$ to $\CC^{g}$:
$$
\mathsf{p}(t) = (p_{1}(t), \dots, p_{g}(t)) \quad \text{and} \quad \mathsf{v}(s) = (v_{1}(s), \dots, v_{g}(s)).
$$ 
Then $\mathsf{p} = \mathsf{v} \circ \mathsf{s}$ by Girard's formula \ref{Albert}. If $1 \leq n \leq g$, then
$$v_{n}(s) = (-1)^{n + 1} n s_{n} + v'_{n}(s),$$
where $v'_{n}(s)$ depends only of $s_{1}, \dots, s_{n - 1}$. This implies that $J(\mathsf{v})$ is lower triangular, with $n$-th diagonal term equal to $(- 1)^{n + 1} n$. Hence,
$$\det J(\mathsf{v}) = (-1)^{[g/2]} \, g!$$
On the other hand,
$$
J(\mathsf{p}) =
\begin{pmatrix}
1               & 1               & \dots & 1 \\
\dots           & \dots           & \dots & \dots \\
k t_{1}^{k - 1} & k t_{2}^{k - 1} & \dots & k t_{n}^{k - 1}\\
\dots           & \dots           & \dots & \dots \\
g t_{1}^{g - 1} & g t_{2}^{g - 1} & \dots & g t_{n}^{g - 1}
\end{pmatrix}.
$$

Then $J(\mathsf{p}) = D.V(t)$, where $D$ is the diagonal matrix $\diag(1, 2, \dots, g)$, and
$$
V(t) =
\begin{pmatrix}
1             & 1             & \dots & 1 \\
t_{1}         & t_{2}         & \dots & t_{n} \\
\dots         & \dots         & \dots & \dots \\
t_{1}^{g - 1} & t_{2}^{g - 1} & \dots & t_{n}^{g - 1}
\end{pmatrix}
$$
is the Vandermonde matrix. Hence,
$$
\det J(\mathsf{p}) = g! \det V(t) = g! \prod_{j < k}(t_{k} - t_{j}),
$$
and since $J(\mathsf{p}) = J(\mathsf{v}).J(\mathsf{s})$, we get \eqref{CHG1}, which implies \eqref{CHG2}. Then \eqref{CHG3} comes from the fact that $\mathsf{s}$ is injective on the open subset $C_{g}$ of $U_{g}$.
\end{proof}

The \emph{bezoutian} of $h_{t}$ is the matrix
$$
B(t) = V(t).{^{t}V(t)} =
\begin{pmatrix}
p_{0}     & p_{1} & \dots & p_{g - 1} \\
p_{1}     & p_{2} & \dots & p_{g} \\
\dots     & \dots & \dots &\dots \\
p_{g - 1} & p_{g} & \dots & p_{2 g - 2}
\end{pmatrix}
\in \Mat_{g}(\RR),
$$
in such a way that $\det B(t) = D_{0}(t)$.

\begin{lemma}
\label{Sylvester}
Let $h_{t} \in \RR[u]$. The following conditions are equivalent:
\begin{enumerate}
\item
\label{HER1}
The roots of $h_{t}$ are real and simple, i.e. $\mathsf{s}(t) \in \Pi^{\circ}_{g}$. 
\item
\label{HER2}
The bezoutian $B(t)$ is positive definite.
\end{enumerate}
\end{lemma}

\begin{proof}
This is a particular case of a theorem of Sylvester, which states that the number of real roots of $h_{t}$ is equal to $p - q$, where $(p,q)$ is the signature of the real quadratic form
$$
Q(x) = {^{t}x}.B(t).x,
$$
where $x = (x_{0}, \dots, x_{g - 1}) \in \RR^{g}$. Here is a short proof: if $1 \leq j \leq g$, define the linear form
$$
L_{j}(x) = x_{0} + x_{1} t_{j} + \dots + x_{g - 1} t_{j}^{g - 1}.
$$
Then $x.V(t) = (L_{1}(x), \dots, L_{g}(x))$, and
$$
Q(x) = \sum_{j = 1}^{g} L_{j}(x)^{2}.
$$
If $t_{j} \in \RR$, the linear form $L_{j}$ is real. If $t_{j} \notin \RR$, the non-real linear form $L_{j} = A_{j} + i B_{j}$ appears together with its conjugate, and
$$
L_{j}^{2} + \bar{L}_{j}^{2} = 2 A_{j}^{2} - 2 B_{j}^{2}.
$$
This shows that if $h_{t}$ has $r$ real roots and $s$ couples of non-real roots, the signature of $Q$ is $(r + s, s)$.
\end{proof}

The bezoutian $B(t)$ is positive definite if and only if its principal minors
$$
M_{j}(p_{1}, \dots, p_{g}) =
\begin{vmatrix}
p_{0} & p_{1} & \dots & p_{j - 1} \\
p_{1} & p_{2} & \dots & p_{j} \\
\dots & \dots & \dots &\dots \\
p_{j - 1} & p_{j} & \dots & p_{2 j - 2}
\end{vmatrix}
\quad (1 \leq j \leq g)
$$
are $> 0$, see for instance \cite[Prop. 3, p. 116]{BkiALG9}. By substituting to the power sums their expression given by Girard's formula of Proposition \ref{Albert}, we obtain $g$ polynomials
$$m_{j} = M_{j} \circ \mathsf{v} \in \ZZ[s_{1}, \dots, s_{g}] \quad (1 \leq j \leq g).$$
Of course, $m_{1} = g$, and
$$
M_{g}(p_{1}, \dots, p_{g}) = \det B(t) = D_{0}(t),
$$
hence, $m_{g}(s) = d_{0}(s)$.  As a consequence of Lemma \ref{Sylvester}, we obtain
$$
\Pi_{g}^{\circ} = \set{s \in \RR^{g}}{m_{j}(s) > 0 \ \text{if} \ 2 \leq j \leq g},
$$
hence:
\begin{lemma}
\label{Hermite}
We have
$$
\Pi_{g} = \set{s \in \RR^{g}}{m_{j}(s) \geq 0 \ \text{if} \ 2 \leq j \leq g}. \rlap \qed
$$
\end{lemma}

\begin{example}
\label{exampled2}
If $g = 2$, then $d_{0}(s) = s_{1}^{2} - 4 s_{2}$, and
$$
\Pi_{2} = \set{s \in \RR^{2}}{d_{0}(s) \geq 0}.
$$
\end{example}

\begin{example}
\label{exampled3}
If $g = 3$, then
$$
d_{0}(s) = s_{1}^{2} s_{2}^{2} - 4 s_{2}^{3} -
4 s_{1}^{3} s_{3} + 18 s_{1} s_{2} s_{3} - 27 s_{3}^{2},
$$
and $m_{2}(s) = 2(s_{1}^{2} - 3 s_{2})$. But if $d_{0}(s) \geq 0$, then $m_{2} \geq 0$. Actually, if
$$
p = - \frac{s_{1}^{2} - 3 s_{2}}{3}, \quad
q =  \frac{2 s_{1}^{3} - 9 s_{1} s_{2} + 27 s_{3}}{27},
$$
then
$$
d_{0}(s) = - (4 p^{3} + 27 q^{2}), \quad m_{2} = - 6 p.
$$
If $d_{0}(s) \geq 0$, then $4 p^{3} \leq - 27 q^{2}$ and $p \leq 0$. Hence, as it is well known, $\Pi_{3}$ is defined by only one inequality:
$$
\Pi_{3} = \set{s \in \RR^{3}}{d_{0}(s) \geq 0}.
$$
\end{example}

\section{The symmetric alcove}
\label{sec_alcove}

The \emph{symmetric alcove} is the compact set
$$
\Sigma_{g} = \mathsf{s}(I_{g}) \subset \Pi_{g}.
$$
We have $\mathsf{s}(\bar{A}_{g}) = \mathsf{s}(I_{g})$, and the induced map
$$
\begin{CD}
I_{g}/\mathfrak{S}_{g} & @>{\sim}>> & \Sigma_{g}
\end{CD}
$$
is a homeomorphism, leading to the commutative diagram
$$
\xymatrix{& I_{g} \ar[dr]^{\pi} \ar[dd]^{\mathsf{s}} \\
\bar{A}_{g} \ar[ur]^{\iota} \ar[dr]_{\simeq} &
& I_{g}/\mathfrak{S}_{g} \ar[dl]^{\simeq} \\
& \Sigma_{g}}
$$
If $p \in \CC[t_{1}, \dots, t_{g}]$ is a symmetric polynomial and if $\lambda \in \CC$, define
$$
p(\lambda ; t) = p(\lambda + t_{1}, \dots, \lambda + t_{g}), \quad t = (t_{1}, \dots, t_{g}).
$$
The polynomial $p(\lambda ; t)$ is symmetric with respect to $t$.

\begin{lemma}
\label{Surjection}
If $t \in \RR^{g}$ and $\lambda > 0$, the following conditions are equivalent:
\begin{enumerate}
\item
\label{SUR1}
$s_{i}(\lambda ; t) > 0$ and $s_{i}(\lambda ; -t) > 0$  for $1 \leq i \leq g$.
\item
\label{SUR2}
$\card{t_{i}} < \lambda$ for $1 \leq i \leq g$. \qedhere
\end{enumerate}
\end{lemma}

\begin{proof}
It suffices to prove the following result: if $t \in \RR^{g}$, the following conditions are equivalent:
\begin{enumerate}
\item
\label{DES1}
$s_{i}(t) > 0$ for $1 \leq i \leq g$.
\item
\label{DES2}
$t_{i} > 0$ for $1 \leq i \leq g$.
\end{enumerate}
If $h_{t} \in \RR[u]$ is defined as in \eqref{roots}, namely
$$
h_{t}(u) = (u - t_{1}) \dots (u - t_{g}) = u^{g} - s_{1} u^{g - 1} + \dots + (- 1)^{g} s_{g},
$$
let $f^{\flat}(u) = (- 1)^{g}f(-u)$. Then
$$
f^{\flat}(u) = (u + t_{1}) \dots (u + t_{g}) = u^{g} + s_{1} u^{g - 1} + s_{2} u^{g - 2} + \dots + s_{g},
$$
and if \eqref{DES1} is satisfied, the roots of $(- 1)^{g}f(-u)$ are $< 0$, and this implies \eqref{DES2}. The converse is trivial.
\end{proof}

The polynomial $s_{i}(\lambda;t) \in \CC[t_{1}, \dots, t_{g}]$ is a linear combination of elementary symmetric polynomials of degree $\leq i$:

\begin{lemma}
\label{FormesAff}
If $1 \leq i \leq g$ and if $\lambda > 0$, then
$$
s_{i}(\lambda;t) = L^{+}_{i}(\lambda ; \mathsf{s}(t)),
$$
with
$$
L^{+}_{i}(\lambda ; s) =
\sum_{k = 0}^{i} \binom{g - i + k}{k} s_{i - k} \lambda^{k},
$$
which is a linear form with respect to $s_{1}, \dots, s_{i}$. Similarly,
$$s_{i}(\lambda;-t) = L^{-}_{i}(\lambda ; \mathsf{s}(t)),$$
where
$$
L^{-}_{i}(\lambda ; s_{1},s_{2}, \dots, s_{i}) = L^{+}_{i}(\lambda ; - s_{1}, s_{2}, \dots, (-1)^{i} s_{i}).
$$

\end{lemma}

\begin{proof}
The Taylor expansion of $h_{t}(u - \lambda)$ with respect to $u$ shows that
$$
s_{i}(\lambda,t) = (-1)^{i} \, \frac{h_{t}^{(g - i)}(-\lambda)}{(g - i)!} \, .
$$
The first formula is obtained by transferring these equalities in the Taylor expansion of $h_{t}^{(g - i)}(\lambda)$ at $0$ :
$$
h_{t}^{(g - i)}(- \lambda) = \sum_{k = 0}^{i} (-1)^{k} h_{t}^{(g - i + k)}(0) \frac{\lambda^{k}}{k!} \, .
\mbox{\qedhere}
$$
The second formula is deduced from the first by noticing that $s_{i}(- t) = (-1)^{i} s_{i}(t)$.
\end{proof}

Considering that $s_{0} = 1$, we have
\begin{align*}
s_{1}(\lambda;t) & = s_{1}(t) + g \lambda, \\
s_{2}(\lambda;t) & = s_{2}(t) + (g - 1) \lambda s_{1}(t) + \frac{g(g - 1)}{2} \, \lambda^{2},\\
s_{g}(\lambda;t) & = \sum_{k = 0}^{g} s_{g - k}(t) \lambda^{k} = (-1)^{g} h_{t}(-\lambda) =
\prod_{i = 1}^{g}(t_{i} + \lambda).
\end{align*}
Hence
$$
L^{\pm}_{1}(2;s) = \pm s_{1} + 2 g, \quad \quad
L^{+}_{2}(2;s) = s_{2} \pm 2 (g - 1) s_{1} + 2 g(g - 1), \\
$$
and
\begin{equation}
\label{lineaire}
L^{\pm}_{g}(2;s) = \sum_{k = 0}^{g} 2^{k} s_{g - k}, \quad \quad
L^{-}_{g}(2;s) = \sum_{k = 0}^{g} (-1)^{g - k} 2^{k} s_{g - k}.
\end{equation}

From Lemmas \ref{Surjection} and \ref{FormesAff} we obtain

\begin{lemma}
\label{DefTheta}
Assume $t \in \RR^{g}$. Then $t \in I_{g}$ if and only if $\mathsf{s}(t) \in \Theta_{g}$, where
$$
\Theta_{g} = \set{s\in \RR^{g}}
{L^{\pm}_{i}(2 ; s) \geq 0 \ \text{for} \ 1 \leq i \leq g}. \rlap \qed
$$
\end{lemma}

Notice that the polyhedron $\Theta_{g}$ is unbounded.

\begin{theorem}
\label{Image}
If $\Sigma_{g} = \mathsf{s}(I_{g})$, then
$$\Sigma_{g} = \Theta_{g} \cap \Pi_{g},$$
where $\Pi_{g}$ and $\Theta_{g}$ are defined in Lemmas \ref{Hermite} and \ref{DefTheta}. Moreover, $\Sigma_{g} = \mathsf{s}(\bar{A}_{g})$ is a semi-algebraic set homeomorphic to the $g$-dimensional simplex, and
$$
\Sigma_{g} \subset \prod_{i = 1}^{g} \left[ - 2^{i} \binom{g}{i}, 2^{i} \binom{g}{i} \right].
$$
\end{theorem}

\begin{proof}
By definition, $\Pi_{g} = \mathsf{s}(\RR^{g})$, hence, the first statement follows from Lemma \ref{DefTheta}. The properties of $\Sigma_{g}$ follow from Proposition \ref{IsoViete}, and the last statement is just a consequence  of the definition of $s_{1}(t), \dots, s_{g}(t)$. 
\end{proof}

\begin{example}
\label{Sigma2}
If $g = 2$, then $d_{0}(s) = s_{1}^{2} - 4 s_{2}$, and
$$
\Pi_{2} = \set{s \in \RR^{2}}{d_{0}(s) \geq 0},
$$
as we saw in Example \ref{exampled2}. The triangle $\Theta_{2}$ is defined by four inequalities:
\begin{align*}
L_{1}^{\pm}(2,s) & = \pm s_{1} + 4 \geq 0, \\
L_{2}^{\pm}(2,s) & = s_{2} \pm 2 s_{1} + 4 \geq 0.
\end{align*}
The symmetric alcove $\Sigma_{2}$ the curvilinear triangle, drawn in Figure \ref{domainsp4}, contained in the square $[-4, 4] \times [-4, 4]$.

\bigskip
\begin{figure}[ht!]
\centering
\framebox[8cm][c]{\includegraphics[scale=0.4]{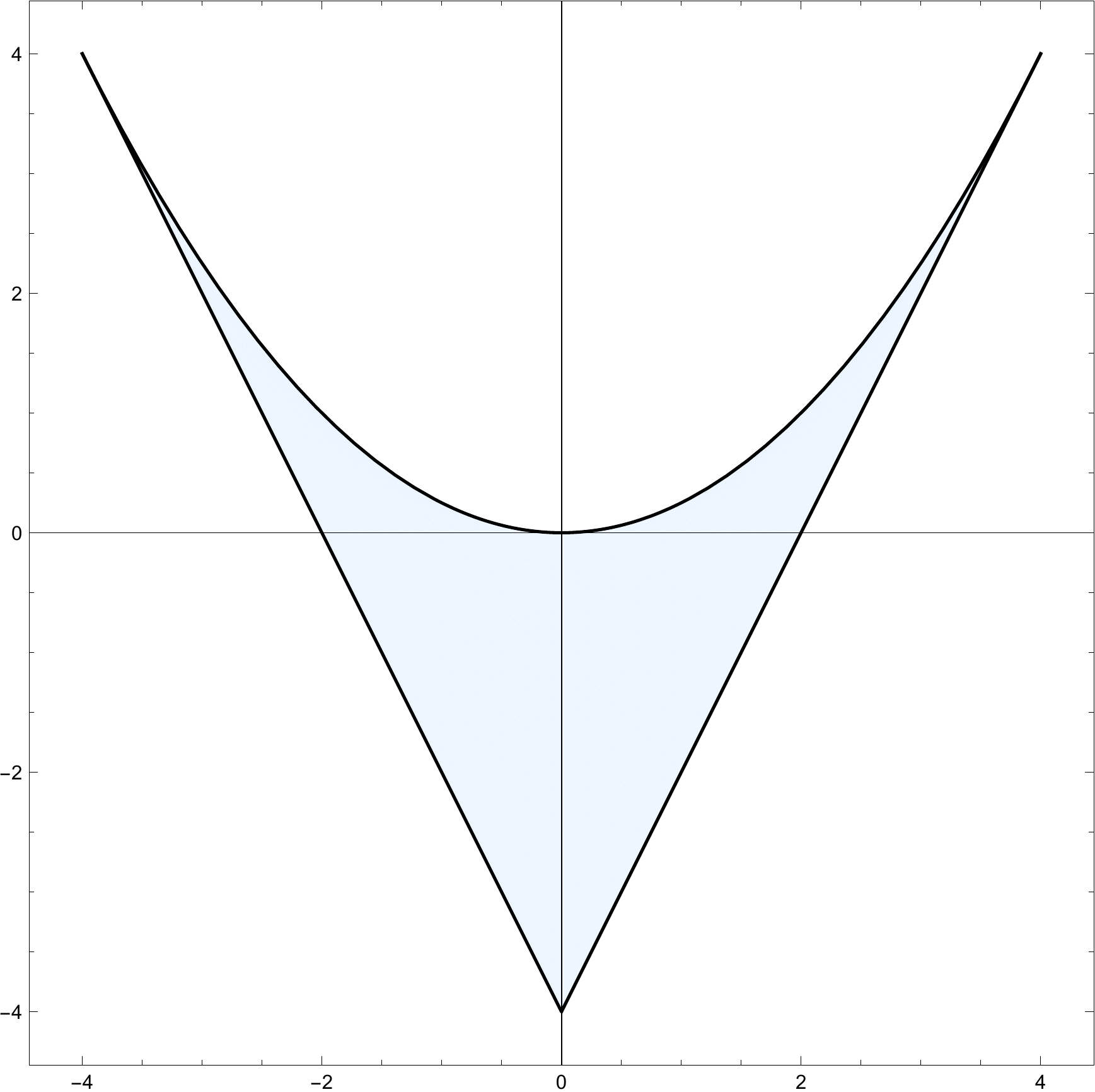}}
\caption{The symmetric alcove $\Sigma_{2}$.}
\label{domainsp4}
\end{figure}
\bigskip

\end{example}

\begin{example}
\label{Sigma3}
If $g = 3$, then
$$
d_{0}(s) = s_{1}^{2} s_{2}^{2} - 4 s_{2}^{3} -
4 s_{1}^{3} s_{3} + 18 s_{1} s_{2} s_{3} - 27 s_{3}^{2},
$$
and
$$
\Pi_{3} = \set{s \in \RR^{3}}{d_{0}(s) \geq 0},
$$
as we saw in Example \ref{exampled3}. The polyhedron $\Theta_{3}$ is defined by six inequalities
\begin{align*}
L_{1}^{\pm}(2,s) & = \pm s_{1} + 6 \geq 0, &
L_{2}^{\pm}(2,s) & = s_{2} \pm 4 s_{1} + 12 \geq 0, \\
L_{3}^{\pm}(2,s) & = \pm s_{3} + 2 s_{2} \pm 4 s_{1} + 8 \geq 0.
\end{align*}
The intersection of $\Theta_{3}$ and of the box
$$[-4, 4] \times [-12,12]  \times [-8,8]$$
make up a polytope $P_{3}$ with $6$ vertices
$$
\begin{array}{lllllllll}
p_{1} = (-6, 12, -8) & = \mathsf{s}(-2,-2,-2), & p_{2} = (-2,-4,8) & = \mathsf{s}(2,-2,-2), \\
p_{3} = ( 2, -4, -8) & = \mathsf{s}( 2, 2,-2), & p_{4} = ( 6,12,8) & = \mathsf{s}( 2, 2, 2), \\
p_{5} = (6, 12, -8), &                & p_{6} = (-6, 12, 8),
\end{array}
$$
and $7$ facets supported the following hyperplanes:
$$
s_{2} = 12, \quad s_{3} = \pm 8, \quad L_{2}^{\pm}(s) = 0, \quad L_{3}^{\pm}(s) = 0. 
$$
Then
$$\Sigma_{3} = \Pi_{3} \cap P_{3}.$$
The symmetric alcove $\Sigma_{3}$ is drawn in Figure \ref{domainsp6}. This set is invariant by the symmetry $(s_{1}, s_{2}, s_{3}) \mapsto (- s_{1}, s_{2}, - s_{3})$.
The graphical representation leads to suppose that 
$$\Sigma_{3} = \Pi_{3} \cap \Delta_{3},$$
where $\Delta_{3}$ is the tetrahedron with vertices $p_{1}, p_{2}, p_{3} , p_{4}$ and support hyperplanes
$$
L_{3}^{\pm}(s) = 0, \quad L_{0}^{\pm}(s) = 0,
$$
where $L_{0}^{\pm}(s) = 24 \pm 4 s_{1} - 2 s_{2} \mp 3 s_{3}$.

\bigskip
\begin{figure}[ht!]
\centering
\framebox[8cm][c]{\includegraphics[scale=0.4]{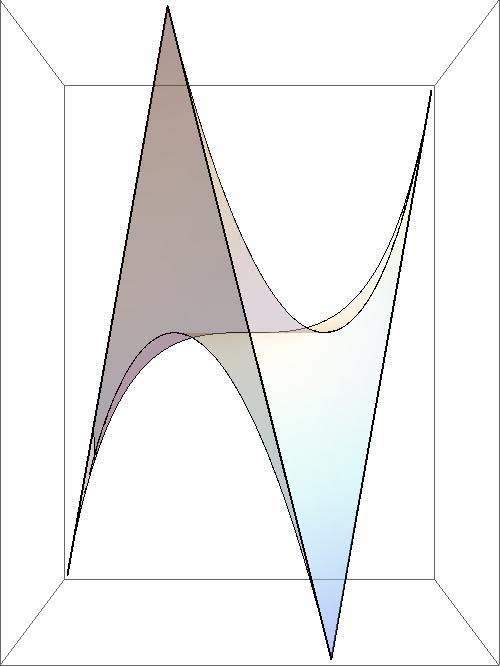}}
\caption{The symmetric alcove $\Sigma_{3}$.}
\label{domainsp6}
\end{figure}
\bigskip

\end{example}

\section{Symmetric integration formula}
\label{sec_symint}

The map $\varphi \mapsto \varphi \circ \mathsf{s}$ defines an isomorphism
$$
\begin{CD}
\mathsf{s}^{*} : \VAS & @>{\sim}>> & \VAI = \VAA
\end{CD}
$$
If $\vaf \in \VAG$, we denote by $\widetilde{\vaf}$ the unique function in $\VAI$ such that
$$
\widetilde{\vaf} \circ \mathsf{s}(t) = \vaf \circ k (t),
$$
that is,
$$
\widetilde{\vaf} \circ \mathsf{s}(2 \cos \theta_{1}, \dots, 2 \cos \theta_{g}) =
\vaf \circ h(\theta_{1}, \dots, \theta_{g}).
$$
Then the map $\vaf \mapsto \widetilde{\vaf}$ is  an isomorphism
$$
\begin{CD}
\VAG & @>{\sim}>> & \VAS.
\end{CD}
$$
inducing by restriction an isomorphism, cf. Proposition \ref{Chevalley} in the appendix:
$$
\begin{CD}
R(G) & @>{\sim}>> & \ZZ[s_{1}, \dots, s_{g}].
\end{CD}
$$
Since $D_{1} \in \ZZ[t_{1}, \dots, t_{g}]^{\sym}$, there is a polynomial $d_{1} \in \ZZ[s_{1}, \dots, s_{g}]$ such that
\begin{equation}
\label{d1sym}
d_{1}(\mathsf{s}(t)) = D_{1}(t) = \prod_{j = 1}^{g} (4 - t_{j}^{2}),
\end{equation}
hence
$$
D_{1}(t) = s_{g}(2;t)s_{g}(2; - t),
$$
where $s_{\pm}(\lambda;t)$ is defined in Lemma \ref{FormesAff}, and
$$
d_{1}(s) = L^{+}_{g}(2 ; s) L^{-}_{g}(2 ; s),
$$
where $L^{\pm}_{g}(2 ; s)$ is defined by \eqref{lineaire}.

\begin{proposition}[Symmetric integration formula]
If $\vaf \in \VAG$, then
$$
\int_{G} \vaf(m) \, dm = \int_{\Sigma_{g}} \widetilde{\vaf}(s) \nu_{g}(s) \, ds,
$$
with $ds = ds_{1} \dots ds_{g}$, and
$$
\nu_{g}(s) = \frac{1}{(2 \pi)^{g}} \sqrt{d_{0}(s) d_{1}(s)},
$$
where $d_{0}(s)$ is given by \eqref{d0sym} and $d_{1}(s)$ by \eqref{d1sym}. 
\end{proposition}

\begin{proof}
By Proposition \ref{change}, we can perform a change of variables from $\Sigma_{g}$ to $A_{g}$, apart from null sets, by putting $s = \mathsf{s}(t)$. If $\varphi \in \VAS$, we have
$$
\int_{\Sigma_{g}} \varphi(s) \frac{ds}{\sqrt{d_{0}(s)}} =
\int_{A_{g}} \varphi(\mathsf{s}(t)) \, dt.
$$
This implies
$$
\int_{\Sigma_{g}} \varphi(s) \nu_{g}(s) \, ds = g!
\int_{A_{g}} \varphi(\mathsf{s}(t)) \lambda_{g}(t).
$$
If $\varphi = \widetilde{\vaf}$, then
$$
\int_{\Sigma_{g}} \widetilde{\vaf}(s) \nu_{g}(s) \, ds =
g! \int_{A_{g}} \vaf \circ k(t) \lambda_{g}(t) \,  dt =
\int_{I_{g}} \vaf \circ k(t) \lambda_{g}(t) \,  dt,
$$
the second equality by using \eqref{weylalgalc}. One concludes with the help of Proposition \ref{weylalg}.
\end{proof}

In other words, if $\phi_{g}$ is the characteristic function of $\Sigma_{g}$, the function $\nu_{g} \, \phi_{g}$ is the joint probability distribution density function of the distribution for the random variables $s_{1}, \dots, s_{g}$.

If $\vat$ is the trace map on $\USp_{2g}$, then
$$\widetilde{\vat}(s)  = s_{1}.$$
One obtains an integral expression of the density by the method of integration along the fibers already used in Remark \ref{Leray}, which reduces here to an application of Fubini's theorem. The linear form $s \mapsto s_{1}$ is a submersion from the open dense subset $U = \mathsf{s}(A_{g})$ of $\Sigma_{g}$ onto $J = (- 2 g, 2 g)$, and if $x \in J$, then 
$$V_{x} = \set{s \in U}{s_{1} = x}$$
is just an intersection with a hyperplane. If
$$
\alpha_{x}(s_{2}, \dots, s_{g}) =
\nu_{g}(x, s_{2}, \dots, s_{g}) \, ds_{2} \wedge \dots \wedge ds_{g},
$$
then
$$
\Phi_{\vat}(z) =
\int_{s_{1} \leq z} \, \nu_{g}(s) ds_{1} \wedge \dots \wedge ds_{g} =
\int_{- 2g}^{z} dx \int_{V_{s}} \alpha_{x}(s_{2}, \dots, s_{g}),
$$
hence:

\begin{proposition}
\label{DensTraceViete}
If $\card{x} < 2 g$, then
$$
f_{\vat}(x) = \int_{V_{x}} \alpha_{x}(s_{2}, \dots, s_{g}). \rlap \qed
$$
\end{proposition}

\begin{example}
If $g = 2$, the symmetric alcove $\Sigma_{2}$ is described in Example \ref{Sigma2}. Here,
$$
d_{0}(s) = s_{1}^{2} - 4 s_{2}, \quad
d_{1}(s) = (s_{2} + 4)^{2} - 4 s_{1}^{2}, \quad
\nu_{2}(s) = \frac{1}{4 \pi^{2}} \, \sqrt{d_{0}(s)d_{1}(s)}.
$$
The graph of $\nu_{2}$ is shown in Figure \ref{DensityNu2}. The maximum of $\nu_{2}$ in $\Sigma_{2}$ is attained at the point
$$
s_{0} = (0, - \frac{4}{3}) , \quad \text{with} \quad
\nu_{2}(s_{0}) = \frac{8}{3 \sqrt{3} \pi^{2}} = 0.155 \dots
$$

\bigskip
\begin{figure}[ht!]
\centering
\framebox[8cm][c]{\includegraphics[scale=0.4]{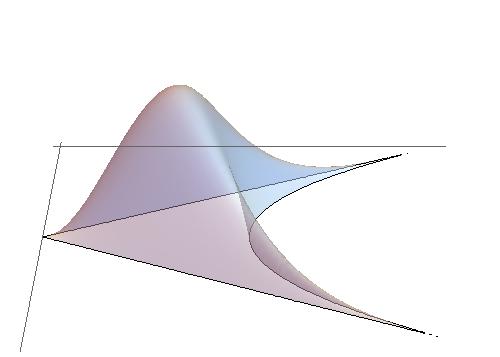}}
\caption{The density $\nu_{2}$.}
\label{DensityNu2}
\end{figure}
\bigskip

By Proposition \ref{DensTraceViete}, we find a definite integral : if $x \geq 0$,
$$
f_{\vat}(x) = \frac{1}{4 \pi^{2}} \int_{2 x - 4}^{x^{2}/4}
\left[ \left( (y + 4)^{2} - 4 x^{2} \right)
\left( x^{2} - 4 y \right) \right]^{\frac{1}{2}} dy.
$$
It can be verified that this formula is in accordance with Theorem \ref{Legendre}.
\end{example}

\begin{example}
If $g = 3$, the symmetric alcove $\Sigma_{3}$ is described in Example \ref{Sigma3}. Here,
\begin{align*}
d_{0}(s) & = s_{1}^{2} s_{2}^{2} - 4 s_{2}^{3} -
4 s_{1}^{3} s_{3} + 18 s_{1} s_{2} s_{3} - 27 s_{3}^{2}, \\
d_{1}(s) & = (2 s_{2} + 8)^{2} - (4 s_{1} + s_{3})^2, \\
\nu_{3}(s) & = \frac{1}{8 \pi^{3}}
 \, \sqrt{d_{0}(s)d_{1}(s)},
\end{align*}
and
$$
V_{x} = \set{s \in \RR^{3}}{d_{0}(s) \geq 0, \ \pm s_{3} + 2 s_{2} \pm 4 x + 8 \geq 0}.
$$
The density is
$$
f_{\vat}(x) = \int_{V_{x}} \nu_{3}(x, s_{2}, s_{3}) ds_{2} ds_{3}.
$$
With this formula in hands, we are able to compute the even moments :
$$1, 1, 3, 15, 104, 909, 9\,449, 112\,398, 1\,489\,410, 21\,562\,086 \dots$$
This sequence is in accordance with the results of \cite[Sec. 4]{Ked-Su} and the sequence A138540 in the OEIS \cite{Sloane}. Actually, it is faster to compute this sequence by noticing that, according to Weyl's integration formula of Proposition \ref{weylalg}, the characteristic function of $\vat$ is given, for $y \in \RR$, by
$$
\widehat{f}_{\vat}(y) = \frac{1}{8 \pi^{3}} \int_{I_{3}} D_{0}(t) \sqrt{D_{1}(t)} \cos(y(t_{1} + t_{2} + t_{3})) \, dt_{1}  dt_{2}  dt_{3}.
$$
An implementation of this integral in the \emph{Mathematica} software gives
\begin{multline*}
\widehat{f}_{\vat}(y) = \\ 24 \left(
- \frac{4 J_{1}(2 y)^{3}}{y^{5}}
+ \frac{11 J_{1}(2 y)^{2}J_{2}(2 y)}{y^{6}}
- \frac{2(3 + y^{2}) J_{1}(2 y) J_{2}(2 y)^{2}}{y^{7}}
+ \frac{5 J_{2}(2 y)^{3}}{y^{6}}
\right),
\end{multline*}
and it suffices to apply \eqref{CaracMoments} to obtain the moments. An approximation of $f_{\vat}$ to any order in $L^{2}([-6,6], dx)$ can be obtained from the sequence of moments, using Legendre polynomials, which form an orthogonal basis of $L^{2}([-1,1], dx)$. For instance the maximum of $f_{\vat}$ is reached for $x = 0$, and we find
$$
f_{\vat}(0) = 0.396\,467 \dots
$$
The graph of $f_{\vat}$ obtained by this approximation process is drawn in Figure \ref{Trace3}.

\begin{figure}[!ht]
\centering
\boxed{\includegraphics[scale=0.5]{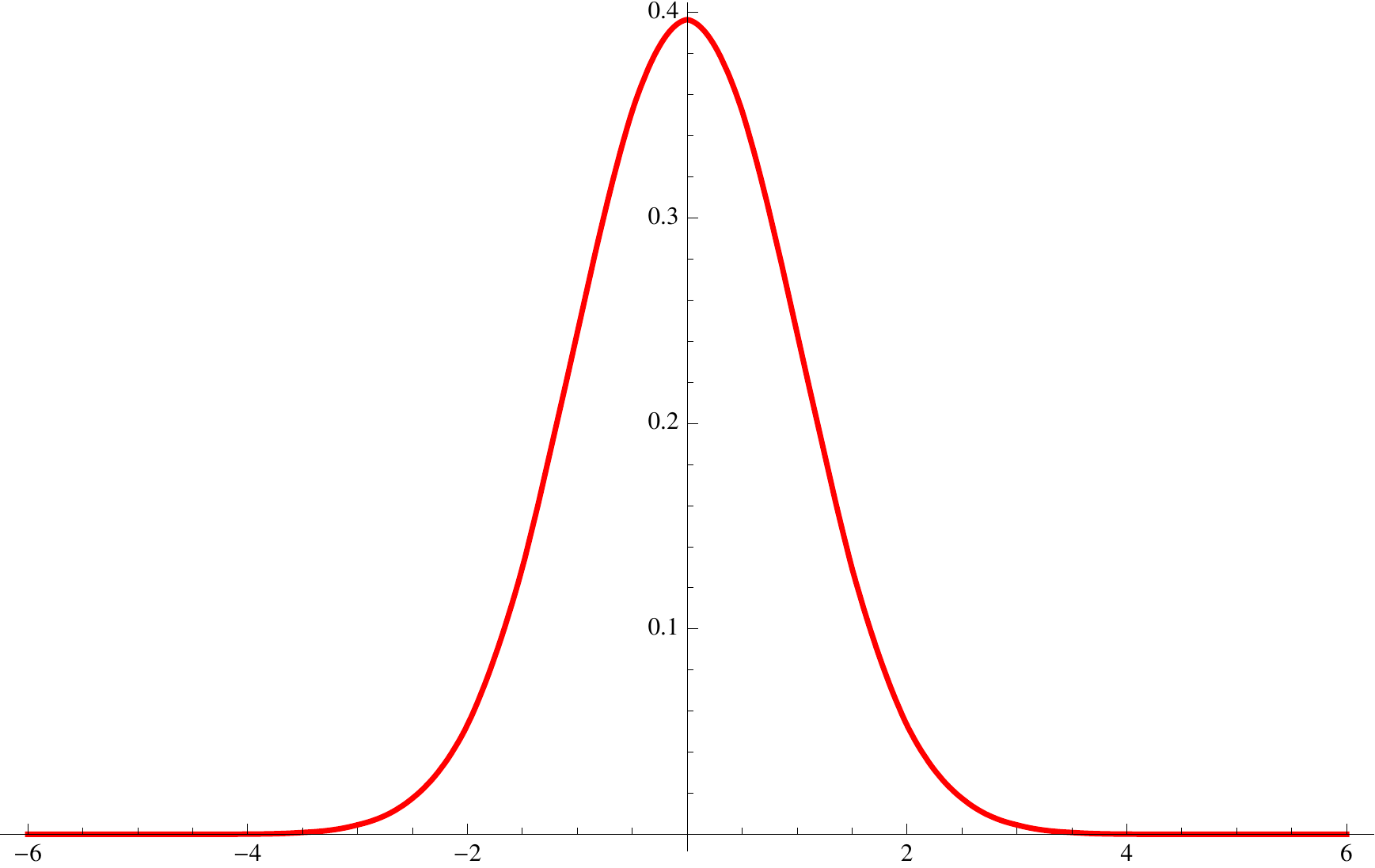}}
\caption{Density of the distribution of $\vat$, case $g = 3$}
\label{Trace3}
\end{figure}
\end{example}

As a final instance, we come back to the case $g = 2$ and apply the symmetric integration formula to the distribution of the character $\vat_{2}$ of the exterior power $\wedge^{2} \pi$ of the identity representation $\pi$ of $\USp_{4}$ on $\CC^{4}$, namely
$$
\vat_{2} \circ h(\theta) = 2 + 4 \cos \theta_{1} \cos \theta_{2}, \qquad
\widetilde{\vat_{2}}(s)  = s_{2} + 2.
$$
The density of $\vat_{2}$ is given by
$$
f_{\vat_{2}}(x) = \int_{I_{\pm}}
\sqrt{16 + 8x + x^{2} - 4z^{2}} \sqrt{z^{2} - 4x} \ dz,
$$
with
\begin{align*}
I_{-} & = \left(-\dfrac{x + 4}{2},\dfrac{x + 4}{2} \right) & \text{if} & \ - 4 < x < 0, \\
I_{+} & = \left(-\dfrac{x + 4}{2},-2\sqrt{x}\right) \cup \left(2\sqrt{x},\dfrac{x + 4}{2} \right) &
\text{if} & \ 0 < x < 4.
\end{align*}

\bigskip
\begin{figure}[!ht]
\centering
\boxed{\includegraphics[scale=0.4]{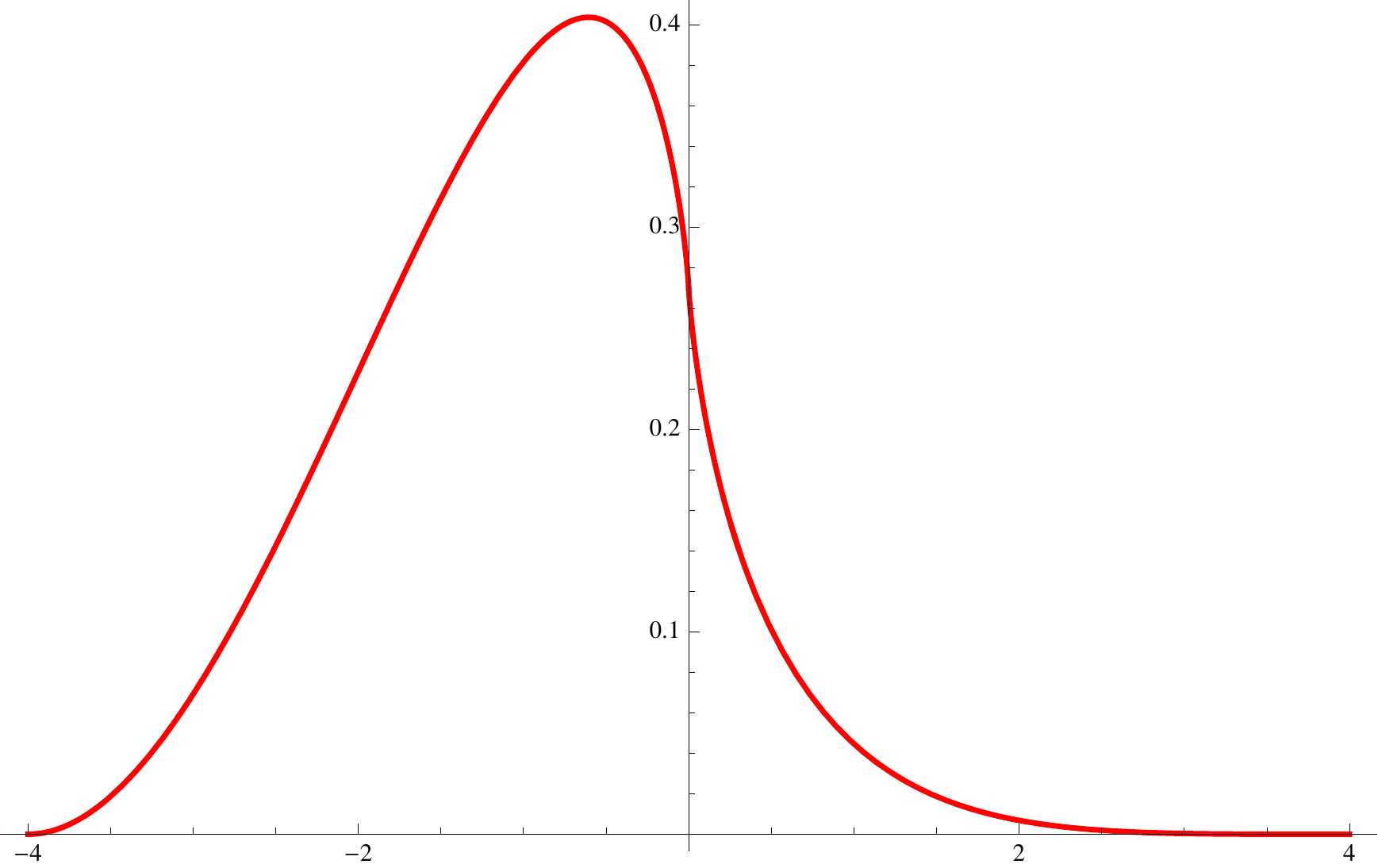}}
\caption{Density of the distribution of $\vat_{2}$, case $g = 2$}
\label{densiteS2}
\end{figure}
\bigskip

The implementation of this integral in the \emph{Mathematica} software gives the sequence of moments (see below), from which one deduces:

\begin{proposition}
\label{ditribS}
Assume $\card{x} < 4$. Then $f_{\vat_{2}}(x)$ is equal to
$$
\frac{\sgn(x)}{24 \pi^{2}}
\left(x (x^{2} - 24 x + 16) E\left(1 - \frac{16}{x^{2}} \right) +
4 (3 x^{2} - 8 x + 48) K\left(1 - \frac{16}{x^{2}} \right) \right). \qed
$$
\end{proposition}

The maximum of $f_{\vat_{2}}$ is reached for $x_{0} = - 0.605\dots$, and $f_{\vat_{2}}(x_{0}) = 0.403 \dots$ Moreover
$$
f_{\vat_{2}}(0) = \frac{8}{3 \pi^{2}} = 0.270 \dots
$$
This function is continuous, but the derivative has a logarithmic singularity:
$$f_{\vat_{2}}'(x) \sim \frac{\log x^{2}}{\pi^{2}}, \quad x \rightarrow 0.$$
The graph of $f_{\vat_{2}}$ is shown in Figure \ref{densiteS2}. The moments $M_{n}$ of $f_{\vat_{2}}$ are obtained by numerical integration:
$$1, -1, 2, - 4,10, - 25, 70, - 196, 588, - 1764 \dots$$
Hence, the random variable $\vat_{2}$ has mean $-1$ and variance $2$. This sequence is, up to sign, the sequence A005817 in the OEIS \cite{Sloane}, such that
$$
M_{2 n}(\vat_{2}) = C_{n} C_{n + 1}, \quad M_{2 n + 1} = - C_{n + 1}^{2},
$$
where
$$
C_{n} = \frac{1}{n + 1}\binom{2 n}{n}
$$
is the $n$th \emph{Catalan number}.

\begin{remark}
\label{SecondTrace}
The representation $\wedge^{2} \pi$ is reducible, cf. for instance Lemma \ref{WeightDecomp} in the appendix. Actually, the two fundamental representations of $\USp_{2}$ are the identity representation $\pi = \pi_{1}$ with character $\vat$ and a representation $\pi_{2}$ of dimension $5$ and character $\boldsymbol{\chi}_{2}$ satisfying
$$
\boldsymbol{\chi}_{2} \circ h(\theta) = 1 + 4 \cos \theta_{1} \cos \theta_{2}, \qquad
\widetilde{\boldsymbol{\chi}_{2}}(s)  = s_{2} + 1.
$$
The representation $\pi_{2}$ is equivalent to the representation corresponding to the morphism of $\USp_{2}$ onto $\SO_{5}$. Since $\vat_{2} = \boldsymbol{\chi}_{2} + 1$, we have
$$\wedge^{2} \pi = \pi_{2} \oplus 1,$$
and $\widetilde{\vat_{2}}(s)  = s_{2} + 2$: this relation is an instance of Theorem \ref{TraceToSym} in the appendix. The random variable $\boldsymbol{\chi}_{2}$, with values on $[-3, 5]$, is standardized, by Remark \ref{standard}. The moments of $f_{\boldsymbol{\chi}_{2}}$ are
$$0,1,0,3,1,15,15,105,190,945 \dots$$
in accordance with the sequence A095922 in the OEIS \cite{Sloane}.
\end{remark}

\appendix

\section{The character ring of $G$}
\label{appendix}

The \emph{character ring} $R(G)$ of $G = \USp_{2g}$ is the subring of $\VAG$ generated, as a $\ZZ$-module, by the characters of continuous representations of $G$ on finite dimensional complex vector spaces. Since every representation of $G$ is semi-simple, the $\ZZ$-module $R(G)$ is free and admits as a basis the set $\widehat{G}$ of characters of irreducible representations of $G$ :
$$R(G) = \sum_{\tau \in \widehat{G}} \ZZ \tau.$$
The \emph{virtual characters} are the elements of $R(G)$, and the characters correspond to the additive submonoid of sums over $\widehat{G}$ with non-negative coefficients. The functions $\theta \mapsto e^{i \theta_{j}} (1 \leq i \leq g)$ make up a basis of the discrete group $\TT^{g}$, and if $h(\theta)$ is as in \eqref{exptore}, the map $h^{*}: f \mapsto f \circ \h$ defines an isomorphism from the group $\mathsf{X}(T)$ of characters of $T$ to the group $\mathsf{X}(\TT^{g})$. Hence, if $\ZZ[\mathsf{X}(T)]$ is the group ring, we have a ring isomorphism
$$
\begin{CD}
h^{*}: \ZZ[\mathsf{X}(T)] & @>{\sim}>> &
\ZZ[\{e^{i \theta_{j}}, e^{- i \theta_{j}}\}].
\end{CD}
$$
Let $\ZZ[\mathsf{X}(T)]^{W}$ be the subring of elements invariants under $W$. Recall that the restriction map
$$
\begin{CD}
R(G) & @>{\sim}>> & \ZZ[\mathsf{X}(T)]^{W}
\end{CD}
$$
is a ring isomorphism \cite[Ch. 9, \S 7, n$^\circ$3, Cor., p. 353]{BkiLIE79en}. From the structure of $W$, we deduce that $h^{*}$ induces a ring isomorphism
\begin{equation}
\label{BasicIsomPoly}
\begin{CD}
R(G) & @>{\sim}>> & \ZZ[2 \cos \theta_{1}, \dots 2 \cos \theta_{g}]^{\sym}.
\end{CD}
\end{equation}
and, by putting $(2 \cos \theta_{1}, \dots 2 \cos \theta_{g}) = (t_{1}, \dots, t_{g})$, a ring isomorphism
\begin{equation}
\label{SecondIsomPoly}
\begin{CD}
R(G) & @>{\sim}>> & \ZZ[t_{1}, \dots t_{g}]^{\sym}.
\end{CD}
\end{equation}
On the other hand, the application $\varphi \mapsto \varphi \circ \mathsf{s}$, where $\mathsf{s}$ is the Vi\`ete map, induces the classical isomorphism
$$
\begin{CD}
\mathsf{s}^{*} : \ZZ[s_{1}, \dots, s_{g}] & @>{\sim}>> & \ZZ[t_{1}, \dots, t_{g}]^{\sym}.
\end{CD}
$$
If $\vaf \in R(G)$, we denote by $\widetilde{\vaf}$ the unique polynomial in $\ZZ[s_{1}, \dots, s_{g}]$ such that
$$
\widetilde{\vaf} \circ \mathsf{s}(2 \cos \theta_{1}, \dots, 2 \cos \theta_{g}) =
\vaf \circ h(\theta_{1}, \dots, \theta_{g}).
$$
We obtain:

\begin{proposition}
\label{Chevalley}
If $G = \USp_{2g}$, the map $\vaf \mapsto \widetilde{\vaf}$ is  a ring isomorphism
$$
\begin{CD}
R(G) & @>{\sim}>> & \ZZ[s_{1}, \dots, s_{g}]. \rlap \qed
\end{CD}
$$
\end{proposition}

Recall from Remark \ref{Phi2g} that $\Phi_{2 g}$ is the set of monic palindromic polynomials of degree $2 g$ in $\CC[u]$ with all roots on the unit circle. We write a typical element of $\Phi_{2 g}$ as
$$
p_{a}(u) = \sum_{n = 0}^{2 g} (- 1)^{n} a_{n} u^{2 g - n},
$$
where $a_{2 g - n} = a_{n}$ for $0 \leq n \leq g$. Moreover $p_{a}(u) = u^{2 g} p_{a}(u^{-1})$. The roots of $p_{a}$ come by pairs : if $p_{a}$ is monic, then
$$
p_{a}(u) = \prod_{j = 1}^{g} (u - e^{i \theta_{j}})(u - e^{- i \theta_{j}}) =
\prod_{j = 1}^{g} (u^{2} - u t_{j} + 1),
$$
with $t_{j} = 2 \cos \theta_{j}$, and the coefficients $a_{n}$ are symmetric polynomials in the variables $\{e^{i \theta_{j}},e^{- i \theta_{j}}\}$, invariant under conjugation.

\begin{theorem}
\label{CoefToSym}
If $t = (t_{1}, \dots, t_{g}) \in \CC^{g}$, and if $0 \leq n \leq 2 g$, define $a_{n}(t)$ by the relation
$$
\prod_{j = 1}^{g} (u^{2} - u t_{j} + 1) = \sum_{n = 0}^{2 g} (- 1)^{n} a_{n}(t) u^{2 g - n}.
$$
If $0 \leq n \leq g$, then
$$a_{n}(t) = \sum_{j = 0}^{n/2} \binom{g + 2 j - n}{j} s_{n - 2 j}(t),$$
where $s_{0}(t) = 1$ and $s_{n}(t)$ is the elementary symmetric polynomial of degree $n$.
\end{theorem}

We deduce Theorem \ref{CoefToSym} from the following lemma.

\begin{lemma}
\label{MapQ}
If $s = (s_{0}, \dots, s_{g}) \in \CC^{g + 1}$, let
$$h_{s}(u) = \sum_{n = 0}^{g} (- 1)^{n} s_{n} u^{g - n},$$
and for $0 \leq n \leq 2 g$, define $q_{n}(s)$ by the relation
$$
u^{g} h_{s}(u + u^{-1}) = \sum_{n = 0}^{2 g} (- 1)^{n} q_{n}(s) u^{2 g - n}.
$$
If $0 \leq n \leq g$, then
$$
q_{n}(s) = \sum_{j = 0}^{n/2} \binom{g + 2 j - n}{j}  s_{n - 2 j}.
$$
\end{lemma}

\begin{proof}
We have
$$
u^{g} h(u + u^{-1}) = u^{g} \sum_{k = 0}^{g} (- 1)^{k} s_{k} (u + u^{-1})^{g - k}.
$$
Since
$$
u^{g} (u + u^{-1})^{g - k} =
u^{g}  \sum_{j = 0}^{g - k} \binom{g - k}{j} (u^{-1})^{g - k - j} u^{j} =
\sum_{j = 0}^{g - k} \binom{g - k}{j} u^{k + 2 j},
$$
one finds
$$
u^{g} h(u + u^{-1}) =
\sum_{j + k \leq g, j \geq 0, k\geq 0} (- 1)^{k} \binom{g - k}{j}  s_{k} u^{k + 2 j}.
$$
Let $k + 2 j = n$. Then $n$ runs over the full interval $[0, 2 g]$ and
$$
j \geq 0 \ \text{and} \ k \geq 0   \ \text{and} \ j + k \leq g
\quad \Longleftrightarrow \quad
j \geq 0 \ \text{and} \ 2 j \leq n \ \text{and} \ j \geq n - g.
$$
Hence, if $1 \leq n \leq 2 g$, we have
$$
q_{n}(s) = \sum_{j = \max(0,n - g)}^{n/2} \binom{g + 2 j - n}{j}  s_{n - 2 j} ,
$$
and the result follows.
\end{proof}

If $p$ is a Weil polynomial and if $p(u) = u^{g}  h(u + u^{-1})$, then $h$ has real roots and is called the \emph{real Weil polynomial} associated to $p$.

\begin{proof}[Proof of Theorem \ref{CoefToSym}]
In Lemma \ref{MapQ}, assume that
$$
h_{s}(u) = \prod_{j = 1}^{g} (u - t_{j}).
$$
Then $s_{n} = s_{n}(t)$, where $t = (t_{1}, \dots, t_{g})$, and
$$u^{g} h_{s}(u + u^{-1}) = \prod_{j = 1}^{g} (u^{2} - u t_{j} + 1).$$
Hence, $a_{n}(t) = q_{n}(1, s_{1}(t), \dots, s_{n}(t))$.
\end{proof}

Define an endomorphism $\mathsf{q}$ of $\CC^{g + 1}$ by 
$$
\mathsf{q}: (s_{0}, \dots, s_{g}) \mapsto (q_{0}(s), \dots, q_{g}(s)).
$$
By Lemma \ref{MapQ}, the square matrix of order $g + 1$ associated to $\mathsf{q}$ is unipotent and lower triangular, with coefficients in $\NN$. For instance:
\begin{equation}
\label{exampleq23}
\text{if} \ g = 2 : \mathsf{q} =
\begin{pmatrix}
 1 & 0 & 0 \\
 0 & 1 & 0 \\
 2 & 0 & 1 \\
\end{pmatrix} ;
\quad \text{if} \ g = 3 : \mathsf{q} =
\begin{pmatrix}
1 & 0 & 0 & 0 \\
0 & 1 & 0 & 0 \\
3 & 0 & 1 & 0 \\
0 & 2 & 0 & 1 \\
\end{pmatrix}.
\end{equation}

\begin{remark}
A reciprocal of the map
$\mathbf{Q}: h \mapsto u^{g}  h(u + u^{-1})$ is constructed as follows. If $n \geq 1$, let $T_{n}(u)$ be the $n$-th \emph{Chebyshev polynomial} \cite[p. 993]{GR}, and $c_{n}(u) = 2 T_{n}(u/2)$, in such a way that
$$
c_{n}(u + u^{- 1}) = u^{n} + u^{- n} \quad \text{if} \ n \geq 1.
$$
Moreover put $c_{0}(u) = 1$. If $p_{a} \in \Phi_{2 g}$ as above and if
$$
[\mathbf{R} p] (u) = \sum_{n = 0}^{g} (- 1)^{n} a_{n} c_{g - n}(u),
$$
it is easy to see that $\mathbf{Q} \circ \mathbf{R}(p) = p$.
\end{remark}

\begin{remark}
It is worthwile to notice that the map $\varphi \mapsto \varphi \circ \mathsf{q}$ defines an isomorphism of the two ring of invariants:
$$
\begin{CD}
\mathsf{q}^{*}: \ZZ[a_{1}, \dots, a_{g}] & @>{\sim}>> & \ZZ[s_{1}, \dots, s_{g}]
\end{CD}
$$
where we identify $\ZZ[a_{1}, \dots, a_{g}]$ and
$\ZZ[a_{1}, \dots, a_{2 g - 1}]/((a_{2g - n} - a_{n}))$.
If we define a polynomial mapping $\mathsf{a}:\CC^{g} \longrightarrow \CC^{g + 1}$ by
$$
\mathsf{a}: t = (t_{1}, \dots, t_{g}) \mapsto (a_{0}(t), \dots, a_{g}(t)).
$$
and if $\mathsf{s}(t) = (s_{0}(t), s_{1}(t), \dots, s_{g}(t))$ is the (extended) Vi\`ete map, then
$$\mathsf{a} = \mathsf{q} \circ \mathsf{s}.$$
These maps are gathered in the following diagram:
\bigskip
$$
\xymatrix{
& \ZZ[\{e^{i \theta_{j}},e^{- i \theta_{j}}\}] \ar@{-}[d]^{W} \ar[rr] &&
\ZZ[t_{1}, \dots, t_{g}] \ar@{-}[d]^{\mathfrak{S}_{g}}\\
R(G): &
\ZZ[a_{1}, \dots, a_{g}] \ar[r]^{\mathsf{q}^{*}} \ar@/_6mm/[rr]_{\mathsf{a}^{*}} &
\ZZ[s_{1}, \dots, s_{g}] \ar[r]^{\mathsf{s}^{*}} &
\ZZ[t_{1}, \dots, t_{g}]^{\sym}}
$$
\end{remark}

We apply the preceding to the character of the $n$-th exterior power $\wedge^{n} \, \pi$ of the identity representation $\pi$ of $G$ in $\CC^{2 g}$. For $0 \leq n \leq 2g$, let
\begin{equation}
\label{TraceExt}
\vat_{n}(m) = \text{Trace} (\wedge^{n} \, m).
\end{equation}
Generally speaking, the representation $\wedge^{n} \pi$ is reducible, and we describe now its decomposition. For each dominant weight $\omega$ of $\USp_{2g}$, we denote by $\pi(\omega)$ the irreducible representation with highest weight $\omega$, cf. \cite{BkiLIE79en}. The following lemma is used in Lemma \ref{AnMean}.
\begin{lemma}
\label{WeightDecomp}
Let $\omega_{1}, \dots, \omega_{g}$ be the fundamental weights of $\USp_{2g}$. Then
\begin{enumerate}
\item
If $1 \leq 2 n + 1 \leq g$, we have
$$
\wedge^{2 n + 1} \pi = \bigoplus_{0 \leq j \leq n} \pi(\omega_{2 j + 1}).
$$
\item
If $2 \leq 2 n \leq g$, we have
$$
\wedge^{2 n} \pi = 1 \oplus \bigoplus_{1 \leq j \leq n} \pi(\omega_{2 j}).
$$
\end{enumerate}
\end{lemma}

\begin{proof}
See \cite[Lemma, p. 62]{Katz01}; the corresponding result for a simple Lie algebra of type $C_{g}$ is proved in \cite[Ch. 8, \S 13, n$^\circ$3, (IV), p. 206-209]{BkiLIE79en}.
\end{proof}

The characteristic polynomial of $m \in \USp_{2 g}$ is
$$
\pc_{m}(u) = \det(u.\I{} - m) =
\sum_{n = 0}^{2 g} (- 1)^{n} \vat_{n}(m)  u^{2 g - n}.
$$
The dual pairing
$$
\begin{CD}
\wedge^n V \times \wedge^{2g-n} V & @>>> & \wedge^{2g} V = \CC
\end{CD}
$$
implies that $\vat_{2g-n} = \vat_{n}$ for $0 \leq n \leq 2g$, and this proves that $\pc_{m} \in \Phi_{2 g}$. If $m$ is conjugate to $h(\theta_{1}, \dots, \theta_{g})$, then
$$
\pc_{m}(u) = \prod_{j = 1}^{g} (u^{2} - u t_{j} + 1),
$$
with $t_{j} = 2 \cos \theta_{j}$, and hence $\vat_{n} \circ k \in \ZZ[t_{1}, \dots t_{g}]^{\sym}$ as expected.

In the notation of Theorem \ref{CoefToSym}, we have
\begin{equation}
\label{IdTraceExt}
\vat_{n} \circ k(t) = a_{n}(t),
\end{equation}
and we deduce from this theorem:

\begin{theorem}
\label{TraceToSym}
Let $m \in \USp_{2 g}$ be conjugate to $k(t_{1}, \dots, t_{g})$. If $0 \leq n \leq g$, then
$$
\vat_{n}(m) = \sum_{j = 0}^{n/2} \binom{g + 2 j - n}{j} s_{n - 2 j}(t).
\rlap \qed
$$
\end{theorem}

For instance, according to \eqref{exampleq23}:

\begin{itemize}
\item
if $g = 2$: \quad $\vat_{2}(m) = s_{2}(t) + 2$ (cf. Remark \ref{SecondTrace}),
\item
if $g = 3$: \quad $\vat_{2}(m) = s_{2}(t) + 3$, \quad $\vat_{3}(m) = s_{3}(t) + 2 s_{1}(t).$
\end{itemize}



\begin{thebibliography}{99}

\bibitem{Arnold}
Arnold, V. I.; Gusein-Zade, S. M.; Varchenko, A. N. \emph{Singularities of differentiable maps. Volume 2. Monodromy and asymptotics of integrals}. Birkh\"auser, New York, 2012.

\bibitem{BkiALG47}
Bourbaki, Nicolas. \emph{Alg\`ebre. Chapitre 4 \`a 7}. Springer, Berlin, 2007. Engl. transl., \emph{Algebra II. Chapters 4--7}. Springer, Berlin, 2003.

\bibitem{BkiALG9}
Bourbaki, Nicolas. \emph{Alg\`ebre. Chapitre 9}. Springer, Berlin, 2007.

\bibitem{BkiLIE79en}
Bourbaki, Nicolas. \emph{Lie groups and Lie algebras. Chapters 7--9}. Springer, Berlin, 2003.

\bibitem{Dickson}
Dickson, Leonard E. \emph{Elementary theory of equations}. Wiley, New York, 1914.

\bibitem{Howe}
DiPippo, Stephen A.; Howe, Everett W. \emph{Real polynomials with all roots on the unit circle and abelian varieties over finite fields}. J. Number Theory 73 (1998), no. 2, 426--450. \emph{Corrigendum},  J. Number Theory 83 (2000), no. 2, 182. 

\bibitem{FKRS}
Fit\'e, Francesc; Kedlaya, Kiran S.; Rotger, V\'ictor; Sutherland, Andrew V. \emph{Sato-Tate distributions and Galois endomorphism modules in genus 2}. Compos. Math. 148 (2012), no. 5, 1390--1442.

\bibitem{Girard}
Girard, Albert. \emph{Invention nouvelle en l'alg\`ebre}. Jan Janssen, Amsterdam, 1629.

\bibitem{GR}
Gradshteyn, I. S.; Ryzhik, I. M. \emph{Table of integrals, series, and products}. Seventh edition. Elsevier/Academic Press, Amsterdam, 2007.

\bibitem{KS}
Katz, Nicholas M.; Sarnak, Peter. \emph{Random matrices, Frobenius eigenvalues, and monodromy}. American Mathematical Society, Providence, RI, 1999.

\bibitem{Katz01}
Katz, Nicholas M. \emph{Frobenius-Schur indicator and the ubiquity of Brock-Granville quadratic excess}. Finite Fields Appl. 7 (2001), no. 1, 45--69.

\bibitem{Ked-Su}
Kedlaya, Kiran S.; Sutherland, Andrew V. \emph{Hyperelliptic curves, $L$-polynomials, and random matrices}. Arithmetic, geometry, cryptography and coding theory, 119--162, Contemp. Math., 487, Amer. Math. Soc., Providence, 2009.

\bibitem{Kohel}
Kohel, David. \emph{Sato-Tate distributions in higher genus}, preprint, 2013.

\bibitem{Macdonald}
Macdonald, Ian G. \emph{Symmetric functions and Hall polynomials}. Oxford University Press, Oxford, 1995.

\bibitem{MOS}
Magnus, W., Oberhettinger, F., Soni, R.P., \emph{Formulas and theorems for the special functions of mathematical physics}, Grund. der math. Wiss., Band 52, Springer, Berlin, 1966.

\bibitem{Sloane}
Sloane, Neil J.A. \emph{The on-line encyclopedia of integer sequences}, 2015.\\
http://www.research.att.com/~njas/sequences/.

\bibitem{Serre2012}
Serre, Jean-Pierre. \emph{Lectures on $N\sb X (p)$}. CRC Press, Boca Raton, 2012.

\bibitem{Watson}
Watson, G. N. \emph{A treatise on the theory of Bessel functions}. Reprint of the second (1944) edition. Cambridge University Press, Cambridge, 1995.

\bibitem{Wolfram}
\emph{The Wolfram functions site}, Wolfram Research, 2015. \\ http://functions.wolfram.com/

\end{thebibliography}
\end{document}